\documentclass[a4paper,12pt]{amsart}
\usepackage{amsmath,amssymb,amscd,amsfonts}
\usepackage{amsrefs, esint}
\numberwithin{equation}{section}
\newtheorem{theorem}{Theorem}[section]

\newtheorem{lemma}[theorem]{Lemma}

\newtheorem{corollary}[theorem]{Corollary}

\def\Re{\mathop{\rm Re}\nolimits}

\def\ddbar{\partial\bar\partial}
\def\d{\partial}

\def\cO{{\mathcal O}}

\def\cR{{\mathcal R}}
\def\cE{{\mathcal E}}
\def\cP{{\mathcal P}}

\def\cM{{\mathcal M}}

\def\cV{{\mathcal V}}

\def\cU{{\mathcal U}}

\def\cH{{\mathcal H}}
\let\ol=\overline

\let\ep=\varepsilon

\def\bC{{\mathbb C}}
\def\bR{{\mathbb R}}

\def\bD{{\mathbb D}}
\def\bE{{\mathbb E}}
\def\bI{{\mathbb I}}
\def\bF{{\mathbb F}}
\def\bH{{\mathbb H}}
\def\a{{\alpha}}
\def\z{{\zeta}}

\def\b{\beta}

\title[Regularization of psh functions]
{Regularization of plurisubharmonic functions with a net of good points}

\author{Long Li}
\address{Institute Fourier, 100 rue des maths 38610 Gi\`eres, Grenoble, France}
\email{Long.Li1@univ-grenoble-alpes.fr}

\begin{document}

\maketitle

\begin{abstract}
The purpose of this article is to present a new regularization technique of quasi-plurisubharmoinc functions 
on a compact K\"ahler manifold. The idea is to regularize the function on local coordinate balls first, 
and then glue each piece together. 
Therefore, all the higher order terms in the complex Hessian of this regularization vanish at the center of each coordinate ball,
and all the centers build a $\delta$-net of the manifold eventually. 
\end{abstract}

\section{Introduction}
Regularization of plurisubharmonic($psh$) functions is an important subject in Several Complex Variables.
It has been widely used in analysis, K\"ahler geometry and algebraic geometry. 
The early known fact about regularization of $psh$ function on a Euclidean ball is that
one can always regularize a $psh$ function $\phi$ by taking convolution with respect to 
some mollifier $\rho_{\ep}$; then $\phi_{\ep}: = \phi * \rho_{\ep} $ is still $psh$ on a smaller ball,
and $\phi_{\ep}$ decreases to $\phi$ while $\ep$ converges to zero. 
Since convolution is commutable with any differential operator in Euclidean space, this implies that we have 
$$ \ddbar \phi_{\ep} = \ddbar\phi* \rho_\ep.$$

However, the situation is not so simple for regularization problems of (quasi)-$psh$ functions on a complex manifold $X$.
In the 90's, Demailly discovered several ways to study this problem on a compact K\"ahler manifold (\cite{Dem83}, \cite{Dem92}), 
and later on a general compact complex manifold (\cite{Dem97}) equipped with some Hermitian metric.
Then Blocki and Kolodziej (\cite{BK}) found a simpler technique of regularization 
for quasi-$psh$ functions with Lelong number zero everywhere.
Moreover, Eyssidieux-Guedj-Zeriahi (\cite{EGZ}) proved another regularization theorem,
by which one can adjust the approximation sequence to continuous quasi-$psh$ functions 
with minimal singularity, provided the cohomology class is big.  

The motivation of this paper is to investigate both the upper bound and the lower bound of the complex Hessian of the regularization.
This leads us to a ``localized" or ``discrete" version of Demailly's technique (\cite{Dem83}, \cite{Dem97}, \cite{Dem92}).
The naive idea is that we can first take convolution locally for a quasi-$psh$ function $\phi$,
and then try to glue each piece together. However, this would not work in general,
and the obstruction exactly comes from the difficulty of combining a ``good glueing" and a ``good Hessian control". 
Therefore, some more delicate analysis (see Sections \ref{sec-ltg}, \ref{sec-hessian}) is necessary to fulfill this idea. 

In fact, the regularization $\phi_\ep$ appearing in Demailly's work (\cite{Dem83}, \cite{Dem97}) has the following Hessian 
\begin{equation}
\label{int-001}
\ddbar \phi_{\ep} =  \ddbar \phi * \rho_{\ep} + \cH + \mathcal{V} + \cR,
\end{equation}
where $\cH$ consists of higher order terms twisted by the curvature of some background metric $\omega$, e.g. $( c_{j\bar kl\bar m} z^l\bar z^m \ddbar \phi) *\rho_{\ep}$ and so on,
and $\cV$ is a term controlled by the Lelong number of $\phi$ and the curvature tensor of $\omega$. 
The remaining term is an error controlled by $O(\ep\log\ep^{-1})$ in (\cite{Dem83}),
and this order has been strengthened to $O(\ep^{N}\log\ep^{-1})$ for arbitrary $N$ in (\cite{Dem97}).
However, in the latter estimate, the information about the higher order term $\cH$ is lost. 

In our case, the global behavior of the approximating function $\phi_{\ep}$ is comparable to Demailly's result.
We also calculated explicit formulas for all the higher order terms $\cH$ with an error control $\cR = O(\ep^2\log\ep^{-1})$.
Moreover, since we are doing convolutions locally first,
it is expected that the complex Hessian $\ddbar\phi_{\ep}$ behaves much better for those points very close to a center.
In fact, we proved (see Theorem (\ref{thm-002}) and Corollary (\ref{cor-001}) for precise statements) that 
there exist a $\ep$-net $\cP_{\ep}$ with the regularization $\phi_{\ep}$, such that the complex Hessian is 
\begin{equation}
\label{000} 
\ddbar\phi_{\ep} = \ddbar \phi * \rho_{\ep} + \mathcal{V},  
\end{equation}
at each point $p\in \cP_{\ep}$. 
That is to say, all higher order terms $\cH$ and the remaining term $\cR$ completely vanish at these finite many points in our regularization.
In fact, equation (\ref{000}) is basically the best situation for which one can hope,
for global regularization of a quasi-$psh$ function on a compact complex manifold.

We hope this new regularization technique could be useful when all important information is concentrated around a net of points,
e.g. the Gromov-Hausdroff limit of a sequence of K\"ahler manifolds. 
Moreover, we also would like to see its application to regularity problems for solutions of certain geometric equation,
e.g. the complex homogeneous Monge-Amp\`ere equation \cite{BD} on a compact K\"ahler manifold with pseudoconvex boundary. 

$\mathbf{Acknowledgement}$: The author is very grateful to Prof. Demailly, for introducing this problem
and lots of useful discussion, and the author also would like to thank Prof. X.X. Chen, Prof. M. P\u aun, and Prof. V. Tosatti 
for further discussion of this paper. Finally, the author wants to thank  Dr. Tao Zheng and Jian Wang for clearing some problems of this paper.

\section{Statement of the Theorem}
\label{sec-stt}

Let $(X,\omega)$ be a K\"ahler manifold with its associated K\"ahler form, 
and a function $\phi$ is called $\omega$-$psh$ if it is upper semi-continuous and satisfies 
$\omega + i\ddbar\phi \geq 0$,
in the sense of current on $X$.

Then we are going to demonstrate the basic idea to regularize this function $\phi$ as follows. 
First, we cover the manifold $X$ by finite many coordinate balls with small radius, 
such that the K\"ahler metric $\omega$ is normal at the center of each ball. 
Then we build the local regularization by taking convolutions of $\phi$ with respect to certain mollifier.
However, the convolution is not taking on a Euclidean ball, but a ball twisted by the metric. 
This will enable us to glue each piece well, and it gives a way to 
compute the complex Hessian for the approximation sequence.

\subsection{Covering construction}
Let $T = i\ddbar\phi$ be a current such that $\varphi$ is a $\omega$-$psh$ function on $X$.
For each point $p\in X$, there is a normal coordinate ball induced by the metric $\omega$
centered at this point with radius $\sqrt{2}\delta$. Then all of these balls form an open covering $\mathcal{U}$ of the manifold $X$.
Here we assume the geodesic distance $\delta$ is much smaller than the injective radius of $X$,
and then these normal coordinates are varying smoothly w.r.t. their centers. 

Thanks to the well known Zorn Lemma, we can select a finite number of elements $\{U_j\}_{1\leq j \leq N}$ in $\cU$, 
such that the following two properties hold:
\begin{enumerate}
\smallskip
\item[(a)] for each pair of centers $(p_j, p_k), j\neq k$, the geodesic distance between them is no smaller than $\sqrt{2}\delta$, i.e.
$d(p_j, p_k) \geq \sqrt{2}\delta $;
\smallskip
\item[(b)] these open sets are a covering of the manifold, i.e. $ X \subset \bigcup_{1\leq j\leq N} U_j $.
\end{enumerate}

In fact, we hope to see an even better covering, such that $(1-\a) U_j$ still forms an open covering of $X$ for any small $\a >0$.
However, this is not clear to be true in general. But the obstruction indeed comes from the local convexity of these geodesic balls. 
Therefore, we claim that we can do the following small surgery on these geodesic balls, in such a way that 
the perturbed geodesic balls $U''_j$ satisfy the following properties: 
\begin{enumerate}
\smallskip
\item[(a')] for each $j$, the distance between the center $p_j$ of the perturbed ball $U''_j$ and its boundary is no less than $1.3\delta$,
i.e. $d(p_j, \d U''_{j}) \geq 1.3\delta$;
\smallskip
\item[(b')] around each $p_j$, there exists a smaller ball $V_j$ centered at $p_j$ with radius $\delta/10$, such that it never intersects with other $U''$ balls
i.e. $V_j \cap U''_k = \emptyset $ for all $k\neq j$;
\smallskip
\item[(c')] these open sets form a covering of the manifold, i.e. $X \subset \bigcup_{1\leq j\leq N} U''_j $.
\end{enumerate}

This is because we can dig some small holes on the boundary of $U_k$ if it is too close to another balls' centers.
Suppose the distance between a center $p_j$ and the boundary $\d U_k$ is smaller than $\delta/10$ for some $k\neq j$
(at worst, $p_j$ is on the boundary $\d U_k$). Then we put $U'_k: = U_k - V_j$, 
where $V_j$ is the geodesic ball centered at $p_j$ with radius $\delta/10$. 
And $U'_k$ with the open sets $\bigcup_{j\neq k}U_k$ still forms an open covering of $X$! 
But we successfully separate a center point $p_j$ with its boundary at least in $\delta/10$ without changing other centers. 
Once we continue this process, it will terminate after finite many steps, and our claim is proved. 

Furthermore, we could switch geodesic distance by Euclidean distance on each (perturbed) normal coordinate ball.
The reason is that these two distances can only differ by a order of $\delta^2$ (Lemma 8.2, \cite{Dem83}) for small enough radius, i.e. 
$$ d(p_j, q) = d_{Euc}(p_j, q) + O(\delta^2), $$
for each center $p_j$ and any point $q\in U''_j$.

\subsection{Truncated metrics}
Now we take a covering $\{ \tilde{U}_j \}$ of slightly larger concentric normal coordinate balls, such that 
there is a local trivialization $\tau_j: \tilde{U}_j \rightarrow B(2\delta)$ for each $j$,
and we set $\varphi_j = \varphi\circ \tau_j^{-1}$ on $B(2\delta)$. 
Let $\xi$ be a tangent vector on $X$ at a point $z\in B(2\delta)$ under the trivialization
$d\tau_j: TX|_{U_j} \rightarrow B(2\delta)\times \bC^n $. 
And the K\"ahler metric $\omega$ introduces a norm for such tangent vector as follows:
$$ || \xi ||_{j,z} = g^{(j)}_{\a\bar\b}(z) \xi^{\a}\bar{\xi}^\b.$$

Since we used normal coordinates on $B(2\delta)$,
the K\"ahler metric has the following Talyor expansion in the ball
\begin{eqnarray}
\label{rp-002}
g^{(j)}_{\a\bar\b}(z) &=& \delta_{\a\b} - c_{\a\bar\b\mu\bar\lambda} z^{\mu}\bar{z}^{\lambda} 
\nonumber\\
&+& e_{\a\bar\b\mu\bar\lambda\gamma}z^{\mu}\bar z^{\lambda}z^\gamma + e_{\a\bar\b\bar\lambda\mu\bar\nu}\bar z^{\lambda}z^{\mu}\bar z^{\nu} + O(|z|^4),
\end{eqnarray}
where the tensor $c_{\a\bar\b\mu\bar\lambda}$ corresponds to the curvature of the metric $\omega_0$ at the origin. 
And we can arrange that the complex conjugate of tensor $e_{\a\bar\b\bar\lambda\mu\bar\nu}$ is $e_{\b\bar\a\lambda\bar\mu\nu}$.
Moreover, all holomorphic or anti-holomorphic indices in the tensor $e_{\a\bar\b\bar\lambda\mu\bar\nu}$ are commutable as follows from the K\"ahler assumption.
Put $$a_{\a\bar\b}(z): =  \delta_{\a\b} - \frac{1}{2}c_{\a\bar\b\mu\bar\lambda} z^{\mu}\bar{z}^{\lambda} +  e_{\a\bar\b\mu\bar\lambda\gamma}z^{\mu}\bar z^{\lambda}z^\gamma  , $$
and we can define the following $n\times n$ matrix of functions as 
$$A(z): = (a_{\a\bar\b}(z)),$$
and notice that this matrix $A$ is Hermitian up to the second order of $z$.   
Moreover, we have the following identity between matrices
\begin{equation}
\label{rp-003}
g(z) = A^*(z)A(z) + O(|z|^4). 
\end{equation}

The next step is to take a smooth cut-off function $\chi$ on $\bR$ such that 
$$ \chi(t): = -\exp \left( \frac{1}{t-1} \right),$$ for $t<1$, 
and $\chi(t) = 0$ for $t\geq 1$. Then we have 
$$ \chi'(t) = \frac{1}{(t-1)^2} \exp \left( \frac{1}{t-1} \right). $$
There is a twisted convolution on $\phi_j$ around any point $z\in B(\sqrt{2}\delta)$ defined as
\begin{equation}
\label{rp-004}
\tilde\phi_j(z,r): = \int_{\eta\in\bC^n}\phi_j(z + r A^{-1}(z)\cdot\eta) \chi'(|\eta|) d\lambda(\eta),
\end{equation}
for  all $0<r<\delta/2$. We will call this $r$ the \emph{radius of the convolution ball}. 
We will prove this convolution still consists of a local $\omega$-$psh$ function up to some small errors in Section (\ref{sec-hessian}),
by computing its complex Hessian.

Next we try to glue $\tilde{\phi}_j$ by taking maxima, in such a way that the glueing process will produce a global quasi-$psh$ function as follows.
For each point $z\in X$, consider a set $\mathcal{A}(z): = \{ j;\ \ z\in U''_j \}$, and then define 
\begin{equation}
\label{rp--005}
\Phi (z) = \max_{j\in\mathcal{A}(z)} \{\tilde\phi_j \}.
\end{equation}
However, this does not always work. The problem arises from the boundary values appearing in the maximum.
Therefore, in order to succeed glueing a set of local $psh$ functions $\{\Phi_j\}$,
we need to require the following condition 
\begin{equation}
\label{rp--006}
\lim_{p \rightarrow \d U''_j} \Phi_j (p) < \max_{k\neq j} \{ \Phi_k (p) \},
\end{equation}
for all $k$ such that the point $p$ stays in the interior of $U''_k$. 
In other words, the maximum value at each point $p$ should never be obtained by some boundary value of $\tilde\phi_j$. 

Moreover, even if we glue them successfully by taking a maximum,
the resulting function is only continuous since the values of $\Phi_j$ and $\Phi_k$ 
could overlap each other. In order to investigate this problem, we can use the so called $regularized$-$maximum$ operator $\cM_{\tau}$ to smooth them out. 
But this causes another small perturbation of the upper bound of the Hessian.

In order to achieve these goals, we need twist the boundary values a bit.
Define a new quasi-$psh$ function on each $U''_j$ as 
$$\Phi_j(z,r) : = \tilde\phi_j(z,r)  - h(\delta, z)|z|^2, $$ 
where $h(\delta, z)$ is certain smooth function defined on $U''_j$, which converges to zero while $\delta\rightarrow 0$. 
This auxiliary function $h(\delta,z)$ will be determined later in Section (\ref{sec-sub-003}),
and we call the whole term $h(\delta, z)|z|^2$ as the \emph{twisted boundary} for our regularization.

\subsection{Statement}
Recall that the collection of perturbed balls $\{ U''_j \}$ is an open covering of $X$ satisfying conditions (a') and (b'),
and the radius of each such ball is close to $\sqrt{2}\delta$. 
We denote $\cP(\delta): = \{ p_j \}_{1\leq j\leq N}$ by the collection of all centers of such balls,
and then $\cP(\delta)$ forms a $\delta$-net of the manifold. 

On the other hand, the radius of the convolution ball is $r$.
Therefore, in order to make a glueing, it is necessary to specify the relation between these two radius. 
We claim that the glueing will succeed if we pick up $r = O(\delta^3)$, 
and the following regularization holds.

\begin{theorem}
\label{thm-002}
Suppose $\phi$ is a $\omega$-psh function on a K\"ahler manifold $X$. 
Then there exists a family of smooth functions $\Phi_{\delta}$ converging to $\phi$ pointwise as $\delta\rightarrow 0$, 
such that the following properties hold

\begin{enumerate}

\item[(1)] 
For each point $p_j\in\cP(\delta)$,
there exists an open coordiante ball $V_j$ centered at it with uniform size in $j$, 
which can be identified with the Euclidean ball $B(\delta/10)$,
and the complex Hessian of $\Phi_{\delta}$ can be computed at any point $z_0\in V_j$ as 
\begin{eqnarray}
\label{rp-thm-002}
\frac{\d^2 \Phi_{\delta}}{\d z^l \d \bar z^m} (z_0)  s^l \bar s^m &=& \frac{1}{r^{2n}} \int \chi' \frac{\d^2\phi}{\d u^l \d \bar u^m } S^l\bar S^m d\lambda(u)
\nonumber\\
&+& \frac{1}{r^{2n-2}} \int \chi' \frac{\d^2\phi}{\d u^l \d \bar u^m } \cH_{q\bar lm\bar p} s^q\bar s^p d\lambda(u)
\nonumber\\
&-& \frac{1}{r^{2n-2}} \int \chi \frac{\d^2\phi}{\d u^k \d\bar u^j}  (c_{j\bar kl\bar m } - \cE )  
s^l\bar s^m d\lambda(u) + \cR' |s|^2
\nonumber\\
\end{eqnarray}
$$S^l = s^l + P_{j\bar lq}(z)(u^j-z^j) s^q, $$
and the remaining term $\cR'$ is of the order 
$  O(|z_0|^2\log\delta^{-1}) $.
\smallskip
\item[(2)] The global lower bound of the Hessian at any point $p\in X$ can be estimated as  
$$  i\ddbar\Phi_{\delta} \geq  - \Big(1+ O(\delta^2\log\delta^{-1}) + \lambda_{r} \Big)\omega,$$
where $\lambda_r > 0$ is decreasing to a limit $\lambda_{\infty}(p)$ while $r\rightarrow 0$.  
Moreover, this limit $\lambda_{\infty}$ is a constant multiple of the Lelong number $\nu_{\varphi}(p)$, 
and the infimum of the curvature tensor of the metric.
\smallskip
\item[(3)] The global upper bound of the Hessian is also determined by equation (\ref{rp-thm-002}),
except that the remaining term $\cR'$ is of the order  $O(\delta^2(\log\delta^{-1})^3)$. 
\end{enumerate}

\end{theorem}

The explicit formulas of the tensors $P_{j\bar lq}, \cH, \cE$ can be found in equation (\ref{ip-001}) and (\ref{ip-008}):
the tensor $P_{j\bar lq}$ is a polynomial of $z$ up to order $4$, $\cH_{q\bar lm\bar p}$ is of the order $|z|^2$,
and $\cE$ is of the order $|z|$.

Comparing with previous works (\cite{Dem83}, \cite{Dem92}, \cite{Dem97}), 
our higher order terms and the remaining term $\cR'$ enjoy a new feature:  
they converge to zero faster and faster when the point $z_0$ is closer and closer to the center $p_j$. 
In particular, we have the following 

\begin{corollary}
\label{cor-001}
For any $\omega$-psh function $\phi$ on $X$, there exists a $\delta$-net $\cP(\delta)$ of the manifold 
and a sequence of smooth functions $\Phi_{\delta}$ converging pointwise to $\phi$, such that $\Phi_\delta$ is globally quasi-$psh$ as 
$$i\ddbar \Phi_{\delta} \geq - \Big(1+ O(\delta^{2}\log\delta^{-1}) + \lambda_r \Big)\omega,$$ 
and its complex Hessian can be computed as 
\begin{equation}
\label{rp-thm-003}
\frac{\d^2 \Phi_{\delta}}{\d z^l \d \bar z^m}  = \frac{1}{r^{2n}} \int \chi' \frac{\d^2\phi}{\d u^l \d \bar u^m } d\lambda(u)
- \frac{1}{r^{2n-2}} \int \chi \frac{\d^2\phi}{\d u^k \d\bar u^j}  c_{j\bar kl\bar m }   d\lambda(u), 
\end{equation}
at each point $p \in \cP(\delta)$. 
\end{corollary}

In fact, the crucial ingredient in this method is that there exists one K\"ahler form $\omega$ on $X$. 
First we would like to point out that the same regularization holds
if we assume that $\phi$ is a $\omega_0$-$psh$ functions,
where $\omega_0$ is an arbitrary K\"ahler form on $X$(not necessary being in the same class with $\omega$). 

More generally, this technique works for any $\Gamma$-$psh$ function $\varphi$ on a K\"ahler manifold,
where $\Gamma$ is any continuous closed real $(1,1)$ form on $X$. 
In this case, we can utilize a trick developed in Demailly \cite{Dem92} as follows.
First, there exists a homogeneous quadratic function $\gamma_j$ on each ball $\tilde{U}_j$,
such that $\phi_j: = \varphi_j + \gamma_j$ is $psh$ on $\tilde{U}_j$. 
Then the convolution $\tilde\phi_j$ defined as before is $psh$ on each $U''_j$, and we put 
$$ \Psi_j (z,r): = \tilde\phi_j - \gamma_j - h(z,\delta) |z|^2.  $$
Comparing $\Psi_j$ with $\Phi_j$, it is easy to see that 
\begin{equation}
\label{rp-add-010}
\Psi_j - \Phi_j = \tilde\gamma - \gamma = O(r),
\end{equation}
where $\tilde\gamma$ is the convolution taken on $\gamma$. 
Then the complex Hessian of this difference can also be estimate from equation (\ref{rp-thm-002}) as 
\begin{equation}
\label{rp-add-011}
\ddbar (\tilde\gamma - \gamma) = O(r) + O(|z_0|^2 \log\delta^{-1}). 
\end{equation}

The extra term $\tilde\gamma - \gamma$ causes no harm to our estimates in Theorem (\ref{thm-002}) if we put $r= O(\delta^3)$.
Namely, we can still glue $\{ \Psi_j (z,r )\}$ to a global quasi-$psh$ function $\Phi_{\delta}$ as before,
and the complex Hessian has a similar formula (with an extra term $\ddbar (\tilde\gamma - \gamma)$) 
as equation (\ref{rp-thm-002}) locally. This time, the global lower bound of the complex Hessian of $\Phi_\delta$ is changed to 
\begin{equation}
\label{rp-add-012}
\ddbar \Phi_{\delta} \geq -\Gamma - \Big( O(\delta^2\log\delta^{-1}) + \lambda_r \Big)\omega,
\end{equation}
and our Theorem (\ref{thm-002}) and Corollary (\ref{cor-001}) follows in this case.


\section{From local to global}
\label{sec-ltg}
Since we first regularize our (quasi-)$psh$ function $\phi$ locally as in equation (\ref{rp-004}), 
the remaining issue is to glue two pieces together,
by a well known technique involving taking maximum on each intersection point going back to Richberg. 
This requires to estimate the difference between $\tilde\phi_j$ and $\tilde\phi_k$ on the intersection $U''_j\cap U''_k$. 

\subsection{local analysis}
Suppose two coordinate balls $U''_j$ and $U''_k$ intersect with each other. 
Consider the trivialization on a slight larger ball as $\tau_j : \tilde U_j\rightarrow B(2\delta)$
and  $\tau_k: \tilde U_k\rightarrow B(2\delta)$. 
In order to distinguish them, we will use 
$\{z\}$-coordinate on the ball $\tilde{U}_j$, 
and $\{w\}$-coordinate on $\tilde{U}_k$.
Therefore, we have $\tau_j(p)= z$ and $\tau_k(p) = w$ for any point $p\in U_j\cap U_k$.
Then there exists a bioholomorphic map $ \tau: = \tau_j \circ \tau^{-1}_k$ from $w$-ball to $z$-ball,
and we also write $z: = z(w) = \tau (w)$ as a function of $w$.
Now the following transition relation (written in $U_k$ coordinate) holds
\begin{equation}
\label{hd-001}
g^{(j)}_{\mu\bar\nu} (z) \frac{\d z^{\mu}}{\d w^{\a}}\frac{\d\bar z^{\nu}}{\d\bar w^\b} = g^{(k)}_{\a\bar\b} (w). 
\end{equation}
Incorporating with the Taylor expansion (\ref{rp-002}), we have 
\begin{eqnarray}
\label{hd-002}
&&(\delta_{\mu\nu} - c^{(j)}_{p\bar q \mu\bar\nu} z^p \bar z^q 
+ e^{(j)}_{\mu\bar\nu p \bar q r}z^{p}\bar z^{q}z^{r} 
+ e^{(j)}_{\mu\bar\nu \bar q r \bar s}\bar z^{q}z^{r}\bar z^{s} + O(|z|^4) ) (\d\tau)^{\mu}_{\a}(\d\tau)^{\bar\nu}_{\bar\b}
\nonumber\\
 &=& \delta_{\a\b} - c^{(k)}_{p\bar q \a \bar\b}w^p \bar w^q + e^{(k)}_{\a\bar\b p \bar q r}w^{p}\bar w^{q}w^{r} 
 + e^{(k)}_{\a\bar\b\bar q r \bar s}\bar w^{q}w^{r}\bar w^{s} + O(|w|^4).
 \nonumber\\
\end{eqnarray}
Let $\mathcal{O}(\delta)$ denote any complex valued matrix whose all coefficients are of the order $\delta$, 
and then equation $(\ref{hd-002})$ can be written as 
\begin{eqnarray}
\label{hd-003}
\{ A^*(z)A(z) + \mathcal{O}(|z|^4) \}_{\mu\bar\nu} (\d\tau)^{\mu}_{\a}(\d\tau)^{\bar\nu}_{\bar\b} 
\nonumber\\
= \{ B^*(w)B(w) + \mathcal{O}(|w|^4) \}_{\a\bar\b},
\end{eqnarray}
or 
\begin{equation}
\label{hd-004}
(A\cdot \d\tau)^*(A\cdot \d\tau)  = B^*B  +  \mathcal{O}(\delta^4). 
\end{equation}
Since we know $A = I + \mathcal{O}(|z|^2)$ and $B = I +\mathcal{O}(|w|^2)$, 
equation (\ref{hd-004}) implies that we have 
\begin{equation}
\label{hd-006}
\d\tau^* \cdot \d\tau = I + \cO(\delta^2) + \d\tau^* \cO(\delta^2) \d\tau + \cO(\delta^3).
\end{equation}
Suppose the polar decomposition of this matrix is $\d\tau^* = U\cdot P$,
where $U$ is an unitary group and $P$ is semi-positive hermitian. 
Notice that $\d\tau$ are uniformly bounded matrices thanks to the fixed geometry of $X$. 
Then the hermitian part $P$ is also of the form $I + \cO(\delta^2)$ by equation (\ref{hd-006}).
And we claim that such matrices are commutable up to higher order terms. 

\begin{lemma}
\label{hd-lem-001}
If two matrices $A, B$ are of the form $I+\cO(\delta^2)$, then we have $[A,B] = \cO(\delta^4)$. 
\end{lemma}
\begin{proof}
Writing the two matrices as $A = I + P$ and $B=I+Q$ where $P, Q$ are both of the type $\cO(\delta^2)$, 
we have 
\begin{equation}
\label{hd-007}
A\cdot B = I + P + Q + P\cdot Q,
\end{equation}
and then the commutator is 
\begin{equation}
\label{hd-008}
[A,B] = P\cdot Q - Q\cdot P = [P,Q] = \cO(\delta^4).
\end{equation}
\end{proof}
Then we can further prove the following 
\begin{lemma}
\label{hd-lem-002}
Suppose equation (\ref{hd-004}) holds, and then there exists an unitary group $V$, such that we have 
\begin{equation}
\label{hd-0075}
 A\cdot \d\tau - V\cdot B = O(\delta^4)
\end{equation}
\end{lemma}
\begin{proof}
First we assume $A$ and $B$ are Hermitian matrices. 
Recall that the transition matrix $\d\tau$ decomposes into $PU^*$,
and then we can re-write equation (\ref{hd-004}) as 
\begin{equation}
\label{add-001}
PA^2P - U^*B^2U = \cO(\delta^4). 
\end{equation}
Put another Hermitian matrix $H$ as $H: = U^* B U$, 
and then it is easy to see that the matrix $H$ is also of the form $I + \cO(\delta^2)$.
Thanks to Lemma (\ref{hd-lem-002}), if we put $Q:= AP $, then $Q = I + \cO(\delta^2)$ is almost Hermitian up to a order of $\delta^4$, i.e. 
$$ Q^* = Q + \cO(\delta^4), $$

Then we can expand the the following two matrices as 
$$H = I + H_0 + H_1 + H_2;$$
$$Q = I + Q_0 +Q_1 + Q_2,$$ 
where $H_0(Q_0)$ is the term of order $\delta^2$ in the expansion of $H(Q)$, $H_1(Q_1)$ is the third order terms and $H_2, Q_2$ are of order $\cO(\delta^4)$.
Notice that $H_0, H_1, Q_0, Q_1$ are all Hermitian matrices. 
Then we can compare the matrix 
\begin{eqnarray}
\label{add-002}
H^*H &=& (I + H_0+  H_1 + H_2 )^2
\nonumber\\
&=& I + 2H_0 + 2H_1 + \cO(\delta^4)
\end{eqnarray}
with
$$Q^*Q = I + 2Q_0 + 2Q_1 + \cO(\delta^4). $$
Thanks to equation (\ref{hd-004}),
there is no second or third order terms in their difference,
and then we must have $H_0 = Q_0$ and $H_1 = Q_1$. Therefore, we have 
\begin{equation}
\label{add-003}
H - Q  = \cO(\delta^4),
\end{equation}
and we proved equation (\ref{hd-0075}) when $A$ and $B$ are both Hermitian. 

In the general case, we use polar decomposition again to put two semi-positive Hermitian matrices as
$$\tilde A = U_1 \cdot A;\ \ \ \ \tilde B = U_2\cdot B,$$
where $U_1$ and $U_2$ are unitary. Notice that we still have 
$$ (\d\tau)^* \tilde A^* \tilde A (\d\tau) -   \tilde B^*\tilde B = O(\delta^4),$$
and the same argument as before implies the following estimate 
\begin{equation}
\label{hd-0076}
U_1\cdot A\cdot \d\tau -  \tilde V\cdot U_2\cdot B = O(\delta^4),
\end{equation}
where $\tilde V$ is another unitary matrix. 
Then our result follows by putting $V = U_1 \tilde V U_2 $.

\end{proof}

For later use, we shall also investigate the size of the derivatives of the matrix $\d\tau$.
\begin{lemma}
\label{hd-lem-005}
We have the following local estimates in the intersection $U_j\cap U_k$ for each indices $\a, ,\lambda, \gamma$
\begin{equation}
\label{add-0003}
\frac{\d z^{\a}}{\d w^{\lambda} \d w^{\gamma}} = O(\delta); \ \ \ \  \frac{\d w^{\a}}{\d z^{\lambda} \d z^{\gamma}} = O(\delta).
\end{equation}
\end{lemma}
\begin{proof}
It is enough to prove one of the estimates. 
By differentiating equation (\ref{hd-001}) with respect to $w$ variables, we have 
\begin{eqnarray}
\label{add-004}
&&\frac{\d g^{(j)}_{\mu\bar\nu} }{\d z^{\lambda}} \frac{\d z^{\lambda}}{\d w^{\gamma}}
\frac{\d z^{\mu}}{\d w^{\a}} \frac{\d \bar z^{\nu}}{\d\bar w^{\b}}
+ g^{(j)}_{\mu\bar\nu} \frac{\d^2 z^{\mu}}{\d w^{\a} \d w^{\gamma}} \frac{\d\bar z^{\nu}}{\d \bar w^{\b}} 
= \frac{\d g^{(k)}_{\a\bar\b}}{\d w^{\gamma}},
\end{eqnarray}
But notice that the size of $\d g$ is controlled by  
\begin{equation}
\label{add-005}
\frac{\d g_{\a\bar\b}}{ \d z^{\gamma}} = - c_{\a\bar\b\gamma\bar\lambda}\bar z^{\lambda} + O(|z|^2) = O(\delta). 
\end{equation}
Therefore, we have 
\begin{equation}
\label{add-005}
g^{(j)}_{\mu\bar\nu} \frac{\d^2 z^{\mu}}{\d w^{\a} \d w^{\gamma}} \frac{\d\bar z^{\nu}}{\d \bar w^{\b}}  = O(\delta),
\end{equation}
and then our result follows since we know $g_{\mu\bar\nu} = \delta_{\mu\nu} + O(\delta^2)$ and $\d\tau =  ( I + \cO(\delta^2) )\cdot U$
for some unitary group $U$. 
\end{proof}

Let us compare the two convolutions $\tilde\phi_j$ and $\tilde\phi_k$ at an intersection point $p\in U_j \cap U_k$,
where $\tau_j(p)=z\in U_j$, $\tau_k(p)=w\in U_k$.
First notice that we can re-write the integral as 
\begin{equation}
\label{hd-009}
\tilde\phi_k (w,r)= \int_{\z\in\bC^n} \phi_j\circ \tau (w + rB^{-1}(w)\cdot \z) \chi'(|\z|^2) d\lambda(\z).
\end{equation}
Notice that the volume form $\chi' d\lambda$ is a $S^{2n-1}$-invariant measure, and then we have 
\begin{equation}
\label{hd-010}
\tilde{\phi}_j (z,r) = \int_{\z\in\bC^n} \phi_j (z + r A^{-1}(z)\cdot V \z)\chi'(|\z|^2)d\lambda(\z),
\end{equation}
where $V$ is the unitary matrix determined by $A(z)$ and $B(w)$ as in Lemma (\ref{hd-lem-002}).  
Then we can first compare the following distance between two points by putting $\tilde\z = B^{-1}\z$
\begin{eqnarray}
\label{hd-011}
&&\tau(w+B^{-1}\cdot\z) - \tau(w) - A^{-1}\cdot V\z
\nonumber\\
&=& \tau(w+ \tilde\z) - \tau(w) - \d\tau\cdot\tilde\z + (\d\tau\cdot B^{-1} -A^{-1}V ) \z
\nonumber\\
&=& O( |\z|^2 , \delta^4|\z|). 
\end{eqnarray}
This enable us to claim the following excepted estimate 
\begin{equation}
\label{hd-0011}
|\tilde\phi_j(z,r) - \tilde\phi_{k}(w,r) | \leq C (\delta r+O(\delta^4)+ r^2)\log r^{-1},
\end{equation}
for some uniform constant $C$.

In order to prove this, we first fix two points $z_0$ and $w_0$ such that $\tau (w_0) = z_0$.
Denote $\phi_{j,z_0}$ by a linear translation of $\phi_j$ in $z$-coordinate, i.e. 
$$ \phi_{j,z_0}(z): = \phi_j(z_0 + z),$$
and then we have for all $\z\in B(r)$
$$ \phi_k (w_0 +  B^{-1}(w_0)\cdot \z ) = \phi_{j,z_0} \{ \tau(w_0 + B^{-1}\z) - \tau(w_0)  \}. $$ 
Put another linear change of coordinates as $\eta = A^{-1}(z_0)V(z_0) \cdot \z$, 
and then we can define the following map by
$$ F(\eta): = \tau(w_0 + B^{-1}\z) - \tau(w_0), $$
Notice that $F(0)=0$, and it is a biholomorphic map between two balls centered at the origin whose radiuses have the order $r$. 

Thanks to Lemma (\ref{hd-lem-002}) again, the differential(on $\eta$ variable) of this map $F$ at the origin is 
\begin{equation}
\label{hd-012}
\d\tau(w_0) \cdot B^{-1} V^* A= I + \cO(\delta^4),
\end{equation}
and its second derivatives at the origin can be estimated by Lemma (\ref{hd-lem-005}) as 
\begin{equation}
\label{hds-0012}
\frac{\d^2 F^j}{\d \eta^k \d \eta^l }(0) =  O(\delta),
\end{equation}
for all $j, k, l$. 
Moreover, the inverse map of $F$ can be written down as 
$$f(u): = A^{-1}VB \cdot \{ \tau^{-1}( u+ z_0) - \tau^{-1}(w_0) \}.$$
It is easy to see that this inverse map enjoys the same properties as the function $F$,
and then we can change our convolutions as 

\begin{equation}
\label{hd-013}
\tilde\phi_j (z_0 ,r)= \frac{1}{r^{2n}}\int_{\eta\in\bC^n} \phi_{j,z_0} (\eta) \chi'\left(\frac{|\eta|^2}{r^2} \right) d\lambda(\eta),
\end{equation}
and 
\begin{equation}
\label{hd-014}
\tilde\phi_k (w_0 ,r )= \frac{1}{r^{2n}}\int_{\eta\in\bC^n} \phi_{j,z_0} (F(\eta) ) \chi'\left(\frac{|\eta|^2}{r^2} \right) d\lambda(\eta).
\end{equation}
Based on the uniform geometry of the manifold, the Taylor expansion of $f$ near the origin can be written as
\begin{eqnarray}
\label{hd-015}
 \eta^{p} &=& \frac{\d f^p}{\d u^q}(0) u^q + \frac{\d^2 f^p}{\d u^i \d u^q}(0) u^i u^q + O(u^3)
 \nonumber\\
 &=& u^p + \frac{\d^2 f^p}{\d u^i \d u^q}(0) u^i u^q + O(u^3, \delta^4 u)
\end{eqnarray}
Then we have 
\begin{equation}
\label{hd-016}
d\eta^p = du^p + 2\frac{\d^2 f^p}{\d u^i \d u^q}(0) u^i du^q + O(u^2,\delta^4),
\end{equation}
and the volume form is 
\begin{equation}
\label{hd-017}
d\lambda(\eta) = d\lambda(u) \left(1 +  4 \Re \frac{\d^2 f^p}{\d u^p \d u^q}(0) u^q + O(|u|^2, \delta^4) \right).
\end{equation}

Now if we compare equation (\ref{hd-014}) with the following integral 
\begin{equation}
\label{hd-018}
\tilde\psi (w_0,r) : = \frac{1}{r^{2n}}\int_{\eta\in\bC^n} \phi_{j,z_0}(F(\eta)) \chi' \left(\frac{|\eta|^2}{r^2}\right)d\lambda(u),
\end{equation}
then the difference will be controlled by an error like (here we assume $\phi<0$)
\begin{equation}
\label{hd-019}
C(\delta r +O(\delta^4) + r^2)\fint_{B(r)}(-\phi),
\end{equation}
for some uniform constant $C$. 
By a further change of variables, equation (\ref{hd-018}) transforms into 
\begin{equation}
\label{hd-020}
\tilde\psi(w_0,r) = \frac{1}{r^{2n}} \int_{u\in\bC^n} \phi_{j,z_0}(u) \chi'\left( \frac{|f(u)|^2}{r^2}\right) d\lambda(u). 
\end{equation}
Notice that we have the following Taylor expansion for the cut-off function 
\begin{eqnarray}
\label{hd-021}
&&\chi' (r^{-2} |f(u)|^2) - \chi' (r^{-2}|u|^2)
\nonumber\\
&=& r^{-2} \chi''(r^{-2}|u|^2) (|f(u)|^2 - |u|^2) + \frac{1}{2} r^{-4} \chi'''(r^{-2}|u^*|^2)(|f(u)|^2 - |u|^2)^2
\nonumber\\
&=& r^{-2} \chi'' \left( O(\delta^4)|u|^2 + 2\Re \frac{\d^2 f^p}{\d u^i \d u^q}(0) u^i u^q\bar u^p + O(|u|^4)\right) 
\nonumber\\
&+& r^{-4}\chi''' \Big( O(\delta^4) |u|^4 + O(|u|^6) \Big)
\nonumber\\
&=& O(\delta |u|, \delta^4).
\nonumber\\
\end{eqnarray}
Here the length of the vector $|u^*|$ is determined by the value of the function $f$ near the origin, 
but the absolute value of these derivatives of $\chi$ is always uniformly bounded. 
Therefore, the error between $\tilde\psi$ and $\tilde\phi_j$ is again controlled by 
\begin{equation}
\label{hd-022}
C(\delta r +O(\delta^4) + r^2 )\fint_{B(r)}(-\phi).
\end{equation}

However, the Lelong numbers $\nu_{\varphi}$ of the $\omega$-$psh$ function $\varphi$ 
is uniformly bounded on a compact K\"ahler manifold. 
Moreover, there exists a uniform constant $C_1>0$ such that we have 
\begin{equation}
\label{hd-023}
\fint_{B(r)}(-\phi) \leq - C_1 \log r,
\end{equation}
around each point $p\in X$. 
Suppose the radius $r$ of the convolution ball is also in the size of $\delta^3$,
then the difference between $\tilde\phi_j$ and $\tilde\phi_k$ is controlled by the following estimate
\begin{equation}
\label{hd-024}
|(\ref{hd-013}) - (\ref{hd-014})| \leq  - C_2 \delta (r + O(\delta^3)) \log r \leq -C_3 \delta^4\log\delta,
\end{equation}
for some uniform constant $C_2, C_3>0$.

\subsection{glueing}
\label{sec-sub-003}
We are going to glue things together to have a global quasi-plurisubharmonic function by taking maximum among all pieces.
Recall that our convolution $\tilde\phi_j$ is $\omega$-$psh$ on $U''_j$,
and then our glueing target is defined to be 
\begin{equation}
\label{hd-026}
\Phi_j (z,r): = \tilde\phi_j(z,r) - h(\delta,z)|z|^2.
\end{equation}

First we claim that if we pick the auxiliary function as 
$$h(\delta) = -\delta^2 \log\delta,$$
then $r = O(\delta^3)$ will make a successful glueing.
This is because we have 
\begin{eqnarray}
\label{hd-027}
&&\Phi_j(\tau(w),r) - \Phi_k(w,r) 
\nonumber\\
&=& h(\delta)(|w|^2 - |\tau(w)|^2) + O( -\delta r\log r),
\end{eqnarray} 
and if the point $p$ is approaching the boundary $\d U''_j$, the error $|\tau(w)|^2$ tends to a value which is larger than $1.5\delta^2$,
as we perturbed the boundary of the ball. Then we have 
\begin{eqnarray}
\label{hd-028}
&& \Phi_j(\tau(w),r) - \Phi_k(w,r) 
\nonumber\\
&\rightarrow& h(\delta) (|z|^2 - \frac{3}{2}\delta^2) + O(\delta r\log r^{-1})
\nonumber\\
&\leq& -\frac{1}{2}\delta^4\log\delta + O(\delta r\log r^{-1}). 
\end{eqnarray}
Therefore, this error becomes strictly negative when $r = C_5\delta^3$ for some uniform constant $C_5>0$,
and our claim is proved. 

Moreover, in order to cancel the perturbation caused by this auxiliary function near the center, 
we can introduce a cut-off function as follows. 
Let $\rho$ is a standard mollifier supported on the unit ball, 
such that $\rho(z) = 1$ for all $|z|\leq 1/10$, and $\rho(z) =0 $ for all $1/2 \leq |z| $.
Then put $\rho_{\delta} (z): = \rho(\delta^{-2} |z|^2)$,
and we have estimate on its derivatives as 
$$\d\rho_{\delta} = O(\delta^{-1});\ \ \  \ddbar \rho_{\delta} = O(\delta^{-2}), $$ 
on the ball $B(\sqrt{2}\delta)$. 
Therefore, we define the following auxiliary function 
$$\tilde h (z, \delta): = - \Big(1 - \rho_{\delta}(z) \Big) \delta^2 \log\delta.$$

This new choice of auxiliary function also gives a successful glueing. 
This is because $\tilde h(\delta,z) = h(\delta)$ 
when $z$ is close to the boundary $\d U''_j$, and $\tilde h(w,\delta) \leq h(\delta)$ for all $w\in U''_k$. 
Moreover, the whole twisted term $\tilde h(z,\delta)|z|^2$ is complete zero inside the ball $B(\delta/10)$.
Therefore, it contributes nothing to the complex Hessian of $\Phi_j$ in $V_j$. 

On the other hand, for all points outside $V_j$, i.e. $|z| \geq 1/10 $, we have 
\begin{eqnarray}
\label{hd-028}
\ddbar (\tilde h(\delta, z) |z|^2) &=&  h(\delta) \Big\{ |z|^2 \ddbar\rho_{\delta} + 2\Re(\d\rho_{\delta}\cdot |z|) + \rho_{\delta} \Big\}
\nonumber\\
&=& O(\delta^2\log\delta^{-1}),
\end{eqnarray}
and this gives the deserved error in the complex Hessian from the twisted boundary in Theorem (\ref{thm-002}).


\section{Hessian estimate}
\label{sec-hessian}
We are going to compute the complex Hessian of the local convolution $\tilde\phi_j$,
and this follows from a standard calculation (\cite{Dem83}).
However, we have to take care of higher order terms in the Taylor expansion since we need a better control of the error term.

\subsection{The commutator}
Recall that our convolution can be written as 
\begin{equation}
\label{hd-030}
\frac{1}{r^{2n}} \int_{\eta\in \bC^n} \phi(z + A^{-1}(z)\cdot\eta) \chi' \left(\frac{|\eta|^2}{r^2}\right)d\lambda(\eta).
\end{equation}
Put another variable $u: = z + v$, where we separate the variables by taking 
$v: = A^{-1}(z)\cdot \eta$,
and then we have the following Taylor expansion
\begin{equation}
\label{hd-031}
\eta^k = v^k - \frac{1}{2} c_{j\bar kl\bar p} z^{l} \bar z^p v^j + e_{j\bar k l\bar pr}z^l \bar z^p z^r v^j;
\end{equation}
\begin{equation}
\label{hd-032}
\bar\eta^j = \bar v^j - \frac{1}{2} c_{j\bar k l\bar p} z^{l} \bar z^p \bar v^k + e_{j\bar k \bar q l \bar s} \bar z^q z^l \bar z^s \bar v^k,
\end{equation}
where all holomorphic indices or anti-holomorphic indices in tensors $e_{\a\bar\b l\bar pr}$ and $e_{\a\bar \b \bar q l \bar s}$ are commutable.
Moreover, the complex conjugate of the tensor $e_{\a\bar \b l\bar p r}$ is $ e_{\b\bar \a \bar l p \bar r}$. 
Then we can change variables while fixing $z$
\begin{equation}
\label{hd-033}
d\eta^k = dv^k - \frac{1}{2} c_{j\bar k l\bar p}z^l\bar z^p dv^j + e_{j\bar k l\bar pr}z^l \bar z^p z^r dv^j,
\end{equation}
and 
\begin{eqnarray}
\label{hd-034}
d\eta^k\wedge d\bar\eta^k &=& dv^k\wedge d\bar v^k - \frac{1}{2}  c_{j\bar k l\bar p}z^l \bar z^p dv^j\wedge d\bar v^k - \frac{1}{2} c_{k\bar j l\bar p} z^l\bar z^p dv^k\wedge d\bar v^j
\nonumber\\
&+& \frac{1}{4} c_{j\bar k l\bar p} c_{k\bar qr\bar s} z^l z^r \bar z^{p} \bar z^s dv^j \wedge d\bar v^q
\nonumber\\
&+& e_{j\bar k l\bar pr}z^l \bar z^p z^r dv^j \wedge d\bar v^k 
+ e_{k \bar j \bar q l\bar s } \bar z^q z^l \bar z^s dv^k \wedge d\bar v^j
\nonumber\\
&-& \frac{1}{2} \Big( e_{j\bar k l\bar pr} c_{k\bar q i\bar n} z^l \bar z^p z^r z^i \bar z^n  + e_{k\bar q \bar n l\bar s}  c_{j\bar k i\bar p} \bar z^n z^l \bar z^s z^i  \bar z^p  \Big) dv^j \wedge d\bar v^q.
\nonumber\\
\end{eqnarray}
Therefore, the volume form can be computed as 
\begin{eqnarray}
\label{hd-035}
d\lambda(\eta)&=& d\lambda(v) \left\{ 1 - \sum_{k,l,p} c_{k\bar k l\bar p} z^l \bar z^p 
+ \frac{1}{4} \sum_{k,l,p,r,s} c_{k\bar kl\bar p}c_{k\bar k r\bar s} z^l \bar z^p z^r \bar z^s \right.
\nonumber\\
&+& \sum_{k, l, p, r}e_{k\bar k l\bar pr}z^l \bar z^p z^r + \sum_{k, q, l, s}e_{k\bar k \bar q l\bar s} \bar z^q z^l \bar z^s + E_{vol}(|z|^5)
\nonumber\\
&+&\left. \sum_{l,p,r,s} \sum_{k<q} \left(c_{k\bar kl\bar p} c_{q\bar q r\bar s} - \frac{1}{2} c_{q\bar k l\bar p} c_{k\bar q r\bar s} \right) z^l\bar z^p z^r \bar z^s  + O(|z|^6)\right\}
\nonumber\\
&=& d\lambda(v) \Big\{ 1 -  c_{k\bar k l\bar p} z^l \bar z^p  + \frac{1}{2} D_{l\bar p r\bar s}  z^l\bar z^p z^r \bar z^s + 
\nonumber\\
&+& e_{k\bar k l\bar pr}z^l \bar z^p z^r + e_{k\bar k \bar q l\bar s} \bar z^q z^l \bar z^s + E_{vol}(|z|^5) + O(|z|^6) \Big\},
\nonumber\\
\end{eqnarray}
where the $4$-tensor $D$ is defined to be 
\begin{eqnarray}
\label{hd-036}
D_{l\bar p r\bar s}: &=&  2\sum_{k<q} \left(c_{k\bar kl\bar p} c_{q\bar q r\bar s} - \frac{1}{2} c_{q\bar k l\bar p} c_{k\bar q r\bar s} \right)
 + \frac{1}{2} \sum_{k} c_{k\bar kl\bar p}c_{k\bar k r\bar s}
\nonumber\\
&=&  \sum_{k<q} \left(c_{k\bar kl\bar p} c_{q\bar q r\bar s} - \frac{1}{2} c_{q\bar k l\bar p} c_{k\bar q r\bar s} \right) 
+   \sum_{k> q} \left(c_{k\bar kl\bar p} c_{q\bar q r\bar s} - \frac{1}{2} c_{q\bar k l\bar p} c_{k\bar q r\bar s} \right) 
\nonumber\\
&+&   \frac{1}{2}  \sum_{k} c_{k\bar kl\bar p}c_{k\bar k r\bar s} 
\nonumber\\
&=& \sum_{k,q}\left(c_{k\bar kl\bar p} c_{q\bar q r\bar s} - \frac{1}{2} c_{q\bar k l\bar p} c_{k\bar q r\bar s} \right),
\nonumber\\
\end{eqnarray}
and we will postpone the calculation of the fifth order term $E_{vol}$ to next section. 

Let $s = (s_1,\cdots, s_n)\in\bC^n$ be as a tangent vector over a point $z_0 \in U''_j$, 
and we are going the compute the complex Hessian of $\tilde\phi_j$ at this point acting on the vector $s$. 
Observe that the following commutator acts on the smooth measure $\chi' d\lambda(\eta)$ as 
\begin{eqnarray}
\label{hd-037}
&& \left(\frac{\d^2}{\d z^l \d\bar z^m} - \frac{\d^2}{\d u^l \d\bar u^m} \right) ( \chi' d\lambda ) s^l\bar s^m 
\nonumber\\
&=& \Re \left( \frac{\d}{\d\bar z^m} - \frac{\d}{\d \bar u^m}\right) \left( \frac{\d}{\d z^l} + \frac{\d}{\d u^l} \right) (\chi' d\lambda) s^{l}\bar s^m.
\end{eqnarray}
Put two new operators as  
$\nabla_l := \d / \d z^l + \d / \d u^l $ and $\nabla'_{\bar m} = \d / \d\bar z^m - \d / \d \bar u^m$,
then we have 
$$ \nabla_l \eta^k = -\frac{1}{2} c_{j\bar kl \bar p} \bar z^p v^j + 2 e_{j\bar kl\bar pr}\bar z^p z^r v^j; $$

$$ \nabla_l \bar \eta^j =  -\frac{1}{2} c_{j\bar k l\bar p} \bar z^p \bar v^k + e_{j\bar k \bar ql\bar s}\bar z^q \bar z^s \bar v^k;$$

$$ \nabla'_{\bar m} \eta^k = -\frac{1}{2} c_{j\bar k q\bar m} z^q v^j + e_{j\bar k p \bar m r} z^p z^r v^j; $$

\[
\nabla'_{\bar m} \bar \eta^j = -2 \delta_{jm} - \frac{1}{2} c_{j\bar kq\bar m} z^q \bar v^k +  c_{j\bar k q\bar m}z^q\bar z^k 
\]
\[
+ 2 e_{j\bar k \bar m p\bar s} z^p \bar z^s \bar v^k - 2 e_{j\bar m \bar q r\bar s} \bar z^q z^r \bar z^s. \]
Therefore, we have 
\begin{eqnarray}
\label{hd-038}
\nabla_l\{ \chi' d\lambda(\eta) \} &=& r^{-2} \chi'' (\bar\eta^k \nabla_l \eta^k + \eta^j \nabla_l \bar \eta^j) + \chi' \{ \nabla_l d\lambda(\eta) \}
\nonumber\\
&=& \chi'' d\lambda(u)r^{-2} \Big\{ - c_{j\bar k l\bar q} \bar z^q v^j \bar v^k  + 
 \Big(B_{j\bar pl, \bar q r\bar s} +  \frac{1}{2} A_{j\bar p l, \bar qr\bar s} \Big) v^j \bar v^p \bar z^q z^r\bar z^s 
 \nonumber\\
 &+& 2 e_{j\bar k l\bar pr} \bar z^p z^r v^j \bar v^k + e_{j\bar k \bar q l\bar s} \bar z^q \bar z^s v^j \bar v^k
 +E_c(|z|^4 |v|^2) + O(|z|^5|v|^2) \Big\}
 \nonumber\\
 &+& \chi' d\lambda(u) \Big\{ -c_{k\bar k l\bar p}\bar z^p + 2D_{l\bar pr\bar s} \bar z^p z^r\bar z^s 
 \nonumber\\
 &+& 2 e_{k\bar k l\bar pr} \bar z^p z^r + e_{k\bar k\bar q l\bar s} \bar z^q \bar z^s  + \d_l E_{vol} (|z|^4) + O(|z|^5) \Big\},
 \nonumber\\
\end{eqnarray}
where the two tensors are defined as 
$$ A_{j\bar pl, \bar q r\bar s}: = \frac{1}{2} \sum_{k}(c_{j\bar k l\bar q} c_{k\bar pr\bar s} + c_{j\bar k r\bar s}c_{k\bar pl\bar q}); $$
$$ B_{j\bar pl, \bar q r\bar s}: = \sum_k c_{j\bar pl\bar q}c_{k\bar k r\bar s},$$
and the higher order terms $E_c$ and $\d_l E_{vol}$ will be calculated next section. 
Moreover, we have 
\begin{eqnarray}
\label{hd-039}
\nabla'_{\bar m} |\eta|^2 &=& -2 v^m - c_{j\bar k r\bar m} z^r v^j \bar v^k + 2 c_{j\bar p r\bar m} z^r \bar z^p v^j 
\nonumber\\
&+& e_{j\bar k n\bar mr} z^n z^r v^j \bar v^k + 2 e_{j\bar k \bar mn\bar s}z^n \bar z^s v^j \bar v^k
\nonumber\\ 
&-& 2 e_{j\bar m \bar q n\bar s} \bar z^q z^n \bar z^s v^j - 2e_{k\bar m n\bar pr} z^n \bar z^p z^r v^k
+ O(|z|^4|v|, |z|^3|v|^2). 
\end{eqnarray}
Combining equation (\ref{hd-039}) and (\ref{hd-038}), a long computation shows the follows 
\begin{eqnarray}
\label{hd-040}
\nabla'_{\bar m}\nabla_l \{\chi' d\lambda \}&=& d\lambda(u) \frac{\chi'''}{r^4} \Big\{ \bI_3 + \bE_3 + \mathbb{F}_3 + \bD_3
+O(|z|^5|v|^3, |z|^4|v|^4) \Big\}
\nonumber\\
&+& d\lambda(u)\frac{\chi''}{r^2} \Big\{ \bI_2 + \bE_2'+ \bE_2 + \mathbb{D}_2 +  \mathbb{F}_2 + O(|z|^4|v|^2, |z|^5|v|) \Big\}
\nonumber\\
&+& d\lambda(u) \chi' \Big\{ \bI_1 + \bE_1
+ \bD_1  + O(|z|^3) \Big \},
\nonumber\\
\end{eqnarray}
where the terms are 
\begin{align*}
\bI_3: =&  2c_{j\bar kl\bar q}\bar z^k v^j\bar v^q v^m 
-2\left( B_{j\bar pl, \bar qr\bar s} + \frac{1}{2}A_{j\bar pl,\bar qr\bar s}    \right) v^j \bar v^p v^m \bar z^q z^r\bar z^s 
\\
- & 2 c_{j\bar kl\bar q}c_{p\bar mr\bar s} v^j \bar v^k v^p \bar z^q z^r\bar z^s;
\end{align*}
\begin{align*}
\bE_3: = -4 e_{j\bar q l\bar pr} \bar z^p z^r v^j \bar v^q v^m - 2 e_{j\bar q \bar p l\bar s} \bar z^p \bar z^s v^j \bar v^q v^m;
\end{align*}
\begin{align*}
\bD_3: = c_{j\bar kl\bar q}c_{p\bar mr\bar s} v^j \bar v^k v^p \bar v^s z^r \bar z^q 
\end{align*}
\begin{align*}
\bI_2: = &  -c_{j\bar k l\bar m} v^jv^k + 2 c_{j\bar ql\bar m} \bar z^q v^j + 2 c_{j\bar j l\bar p} \bar z^p v^m 
\\
-& 2\left[ (B_{j\bar ml, \bar q r\bar p} + B_{j\bar m r,\bar ql\bar p}) + \frac{1}{2} A_{j\bar ml, \bar qr\bar p}\right] v^j z^r \bar z^q \bar z^p
\\
- & 2D_{l\bar pr\bar s} \bar z^p z^r \bar z^s v^m;
\end{align*}
\begin{align*}
\bE_2: =& -4 e_{j\bar ml\bar pr }\bar z^p z^r v^j -2 e_{j\bar m \bar q l\bar s}\bar z^q \bar z^s v^j 
\\
-& 4 e_{k\bar k l\bar pr}\bar z^p z^r v^m - 2 e_{k\bar k\bar q l\bar s} \bar z^q\bar z^s v^m
\end{align*}
\begin{align*}
\bE'_2: = 2 e_{j\bar k l\bar mr} z^r v^j \bar v^k + 2 e_{j\bar k \bar ml\bar s} \bar z^s v^j \bar v^k;
\end{align*}
\begin{align*}
\mathbb{D}_2: =& \Big(B_{j\bar p l, \bar mr\bar q} +B_{j\bar pl, \bar qr\bar m} + B_{j\bar pr, \bar ml\bar q} \Big) v^j\bar v^p z^r\bar z^q
\\
+& \frac{1}{2} \Big(A_{j\bar pl, \bar mr\bar q}+ A_{j\bar pl,\bar qr\bar m} \Big) v^j\bar v^p z^r\bar z^q;
\end{align*}
\begin{align*}
\bI_1 := - c_{k\bar k l\bar m};
\end{align*}
\begin{align*}
\bE_1: = 2 e_{k\bar k l\bar mr}z^r + 2 e_{k\bar k \bar m l \bar s}\bar z^s,
\end{align*}
\begin{align*}
\bD_1: = ( D_{l\bar m r\bar s} + D_{l\bar s r\bar m})z^r\bar z^s 
\end{align*}
and the higher order terms $\bF_3(|z|^4|v|^3)$ and $\bF_2(|z|^4|v|)$ will be treated in later section.
Notice that above terms, except $\mathbb{D}_i (i=1,2,3)$, could be divergent while $r, \delta$ decreasing to zero even if $r = O(\delta^3)$.
Then we have to use integration by parts to swipe the derivatives into $\phi$!
And we are going to carry out all details for the computation as follows.
We fix the point as $z = z_0$, and do integration by parts to $\bI_i (i =1,2,3)$ terms first
 
\begin{equation}
\label{hd-041}
\frac{\d}{\d \bar v^m} |\eta|^2 = v^m - c_{j\bar m r\bar s} z^r\bar z^s v^j + O(|z|^3|v|).
\end{equation}
Therefore, we have 
\begin{eqnarray}
\label{hd-042}
\frac{\d^2}{\d v^q \d \bar v^m} \{\chi'(r^{-2}|\eta|^2) \} &=& r^{-4} \chi''' \{ \bar v^q v^m - c_{j\bar m r\bar s}z^r\bar z^s v^j\bar v^q 
\nonumber\\
&-& c_{q\bar pr\bar s} z^r\bar z^s \bar v^p v^m + O(|z|^3|v|^2) \}
\nonumber\\
&+& r^{-2} \chi'' \{ \delta_{mq} - c_{q\bar m r\bar s} z^r\bar z^s + O(|z|^3) \},
\end{eqnarray}
and 
\begin{eqnarray}
\label{hd-043}
2\bar z^k c_{j\bar kl\bar q}\frac{\d^2}{ \d v^q \d \bar v^m} \{ v^j \chi'\} &=& 
r^{-4} \chi''' \{ 2 c_{j\bar k l\bar q} \bar z^q v^j \bar v^q v^m - 2 c_{j\bar kl\bar q } c_{p\bar m r\bar s}\bar z^q z^r\bar z^s v^j\bar v^k v^p 
\nonumber\\
&-& 2 c_{j\bar k l\bar q}c_{k\bar p r\bar s}  \bar z^q z^r \bar z^s v^j \bar v^p v^m + O(|z|^4|v|^3) \}
\nonumber\\
&+& r^{-2} \chi'' \{ 2 c_{j\bar k l\bar m} v^j\bar z^k  + 2 c_{j\bar j l\bar q} \bar z^q v^m - 2 B_{j\bar mr, \bar q l\bar p} z^r\bar z^q \bar z^p v^j
\nonumber\\
&-& 2 c_{j\bar k l\bar q} c_{k\bar p r\bar m} z^r \bar z^p \bar z^q v^j +O(|z|^4|v|) \}. 
\end{eqnarray}
Put 
$$ \tilde A_{j\bar pl, \bar qr\bar s}: = \frac{1}{2} \sum_k (c_{j\bar kr\bar s}c_{k\bar pl\bar q} - 3 c_{j\bar kl\bar q}c_{k\bar pr\bar s}),$$
and if we compare equation (\ref{hd-040}) with (\ref{hd-043}), then it becomes as follows
\begin{eqnarray}
\label{hd-044}
\nabla'_{\bar m}\nabla_l \{\chi' d\lambda\} &=& \frac{\d^2}{ \d v^q \d \bar v^m} \{ 2\bar z^k c_{j\bar kl\bar q} v^j \chi'\} d\lambda(u)
\nonumber\\
&+& d\lambda(u) \frac{\chi'''}{r^4} \Big\{ \bE_3 + \bF_3 + \bD_3 + O(|z|^5|v|^3, |z|^4|v|^4)
\nonumber\\
&-&\left( 2B_{j\bar pl, \bar qr\bar s} + \tilde A_{j\bar pl,\bar qr\bar s}    \right) v^j \bar v^p v^m \bar z^q z^r\bar z^s  \Big\} 
\nonumber\\
& +& d\lambda(u)\frac{\chi''}{r^2} \Big\{ -c_{j\bar k l\bar m} v^jv^k 
 - \left(2B_{j\bar ml, \bar q r\bar p}  +  \tilde A_{j\bar ml, \bar qr\bar p}\right) v^j z^r \bar z^q \bar z^p
\nonumber\\
&-& 2D_{l\bar pr\bar s} \bar z^p z^r \bar z^s v^m + \bE_2 +\bE'_2+ \bF_2 + \bD_2 + O(|z|^4 |v|^2, |z|^5|v|) \Big\}
\nonumber\\
&+& d\lambda(u) \chi' \{ -c_{k\bar k l\bar m} +\bE_1+  \bD_1 + O(|z|^3)   \}.
\nonumber\\
\end{eqnarray}

Moreover, we can switch $\chi'(r^{-2}|\eta|^2)$ by $\chi'(r^{-2}|v|^2)$ by adding higher order terms,
since we have the Taylor expansion for an arbitrary smooth cut-off function $\rho$ as 
\begin{eqnarray}
\label{hd-00445}
&&\rho(r^{-2}|\eta|^2) - \rho(r^{-2}|v|^2) 
\nonumber\\
&=& r ^{-2}\rho' (|\eta|^2 - |v|^2)
\nonumber\\
&=& r^{-2} \rho' (c_{j\bar k l\bar p} v^j\bar v^k z^l \bar z^p + O(|z|^3|v|^2)) = O(|z|^2).
\end{eqnarray}
Now we have 
\begin{equation}
\label{hd-00455}
r^2 \frac{\d^2}{\d v^k \d \bar v^j} \{ c_{j\bar kl\bar m}\chi (r^{-2}|v|^2) \} = \frac{\chi''}{r^2} c_{j\bar kl\bar m} v^j\bar v^k + \chi' c_{k\bar kl\bar m},
\end{equation}
and also for $\bE_i (i=1,2,3)$, we have 
\begin{eqnarray}
\label{hd-00466}
&&r^2 \frac{\d^2}{\d v^k \d \bar v^j} \Big\{ (2 e_{j\bar k l\bar mr}z^r + 2 e_{j\bar k\bar ml\bar s} \bar z^s ) \chi (r^{-2}|v|^2) \Big\} 
\nonumber\\
&=& r^{-2} \chi'' \Big(2 e_{j\bar k l\bar mr}z^r + 2 e_{j\bar k\bar ml\bar s} \bar z^s \Big) v^j \bar v^k
\nonumber\\
&+& \chi' \Big(  2 e_{k\bar k l\bar mr} z^r + 2 e_{k\bar k\bar m l\bar s} \bar z^s \Big).
\end{eqnarray}
Therefore, we have 
\begin{eqnarray}
\label{hd-00477}
&&\frac{\chi'''}{r^4} \bE_3 + \frac{\chi''}{r^2}(\bE_2 + \bE'_2) + \chi' \bE_1
\nonumber\\
&= &( - 4 e_{j\bar q l\bar pr} \bar z^p z^r - 2 e_{j\bar q \bar p l\bar s} \bar z^p \bar z^s ) \frac{\d^2}{\d v^q \d\bar v^m} ( v^j \chi')
\nonumber\\
&+ & r^2 \frac{\d^2}{\d v^k \d \bar v^j} \Big\{ (2 e_{j\bar k l\bar mr}z^r + 2 e_{j\bar k\bar ml\bar s} \bar z^s ) \chi (r^{-2}|v|^2) \Big\}.
\end{eqnarray}
And then we have 
\begin{eqnarray}
\label{hd-045}
&&(2B_{j\bar pl, \bar q r\bar s} + \tilde A_{j\bar pl, \bar qr\bar s}) \bar z^q z^r \bar z^s \frac{\d^2}{ \d v^p \d \bar v^m} \{ v^j \chi' (r^{-2}|v|^2)\}
\nonumber\\
&=& \frac{\chi'''}{r^4} \Big\{ ( 2 B_{j\bar pl, \bar qr\bar s} +  \tilde A_{j\bar pl, \bar qr\bar s})  \bar z^q z^r \bar z^s v^j \bar v^p v^m \Big\}
\nonumber\\
&+& \frac{\chi''}{r^2} \Big\{ (2B_{j\bar ml, \bar q r\bar p} + \tilde A_{j\bar ml, \bar qr\bar p} ) \bar z^q z^r \bar z^p v^j
+ (2B_{k\bar k l, \bar pr\bar s} + \tilde A_{k\bar kl, \bar p r\bar s}) \bar z^p z^r \bar z^s v^m \Big\}. 
\nonumber\\
\end{eqnarray}
However, observe the following fact 
\begin{eqnarray}
\label{hd-046}
&& \sum_k (2B_{k\bar kl,\bar pr\bar s} + \tilde A_{k\bar kl, \bar pr\bar s}) 
\nonumber\\
&=& \sum_{k,q} 2c_{k\bar k l\bar p}c_{q\bar qr\bar s} + \frac{1}{2} \sum_{k,q}(c_{k\bar q r\bar s}c_{q\bar kl\bar p} - 3 c_{k\bar ql\bar p}c_{q\bar k r\bar s})
\nonumber\\
&=&  \sum_{k,q} ( 2c_{k\bar k l\bar p}c_{q\bar qr\bar s} -  c_{k\bar q r\bar s}c_{q\bar kl\bar p}  ) = 2D_{l\bar p r\bar s}.
\end{eqnarray}
Therefore, incorporating equations (\ref{hd-00455}), (\ref{hd-045}) and (\ref{hd-046}) into equation (\ref{hd-044}), the commutator further reduces to 
\begin{eqnarray}
\label{hd-047}
\nabla'_{\bar m}\nabla_l \{\chi' d\lambda\}&=& \frac{\d^2}{ \d v^q \d \bar v^m} \Big\{ 2\bar z^k c_{j\bar kl\bar q} v^j \chi' \Big\} d\lambda(u)
\nonumber\\
&-&   \frac{\d^2}{ \d v^q \d \bar v^m} \Big \{(2B_{j\bar ql, \bar p r\bar s} + \tilde A_{j\bar ql, \bar pr\bar s}) \bar z^p z^r \bar z^s v^j \chi' \Big\} d\lambda(u)
\nonumber\\
&-&  \frac{\d^2}{\d v^q \d\bar v^m}  \Big\{ (  4 e_{j\bar q l\bar pr} \bar z^p z^r  + 2 e_{j\bar q \bar p l\bar s} \bar z^p \bar z^s ) v^j \chi' \Big\} d\lambda (u)
\nonumber\\
&-& r^2 \frac{\d^2}{\d v^k \d\bar v^j} \Big\{  (c_{j\bar kl\bar m } - 2 e_{j\bar k l\bar mr}z^r - 2 e_{j\bar k\bar ml\bar s} \bar z^s) \chi  \Big\} d\lambda(u)  
\nonumber\\
&+& \cR(z,v) d\lambda(u),
\nonumber\\
\end{eqnarray}
where the remaining error is
$$ \cR(z,v): = \frac{\chi'''}{r^4} O(|z|^4|v|^3, |z|^2|v|^4)  + \frac{\chi''}{r^2} O(|z|^4|v|, |z|^2|v|^2) + \chi' O(|z|^2). $$

This formula looks very complicate. However, at the center of the ball, we have a rather simple form for $\nabla'_{\bar m} \nabla_l (\chi' d\lambda)$ as 
\begin{equation}
\label{hd-0048}
- r^2 \frac{\d^2}{\d v^k \d\bar v^j} \Big\{  c_{j\bar kl\bar m } \chi  \Big\} d\lambda(u),
\end{equation}
and no remaining terms left at the center. 
\subsection{Higher order terms}
We are going to compute the higher order terms in the integral.
Here higher order terms refer to $|z|^4|v|^3$ terms in the bracket of $r^{-4}\chi'''$ or $|z|^4|v|$ terms in the bracket of $r^{-2}\chi''$.

And we shall use the following convention to simplify the calculation of tensors:
we omit all indices summing over $z$ or $\bar z$ variables. For instance,
the tensor $e_{j\bar k i \bar p r} z^i \bar z^p z^r$ is written as $e_{j\bar k}$,
$e_{j\bar k \bar q r\bar s} \bar z^{q} z^r \bar z^s$ is $\bar e_{j\bar k}$,
$c_{j\bar k l \bar s} \bar z^s$ is $c_{j\bar k, l}$, 
$e_{j\bar k l \bar pr}z^r\bar z^p$ is $e_{j\bar k, l}$ and so on. 
 
First notice that the fifth order term $E_{vol}$ in the volume form $d\lambda(\eta)$ can be calculated from equation (\ref{hd-034}) as $D_{l\bar p r\bar s}$
\begin{eqnarray}
\label{hdd-001}
-2E_{vol}: &=& (\sum_k e_{k\bar k} + \sum_{k} \bar e_{k\bar k}) c_{k\bar k}
\nonumber\\
&+& \sum_{k<q} \Big\{ 2 c_{k\bar k}(e_{q\bar q} + \bar e_{q\bar q}) + 2 c_{q\bar q} (e_{k\bar k} + \bar e_{k\bar k}) 
\nonumber\\
&-& (c_{q\bar k}e_{k\bar q} + c_{k\bar q}e_{q\bar k}) -(c_{k\bar q}\bar e_{q\bar k} + c_{q\bar k}\bar e_{k\bar q}) \Big\}
\nonumber\\
&=& \sum_{k,q} (c_{k\bar k}e_{q\bar q} + c_{q\bar q}e_{k\bar k} ) - \frac{1}{2}(c_{q\bar k}e_{k\bar q} + c_{k\bar q}e_{q\bar k})
\nonumber\\
&+& \sum_{k,q} (c_{k\bar k}\bar e_{q\bar q} + c_{q\bar q}\bar e_{k\bar k})- \frac{1}{2}(c_{k\bar q}\bar e_{q\bar k} + c_{q\bar k}\bar e_{k\bar q}).
\end{eqnarray}

Then a brutal force calculation shows all $\bF_i (i =2,3)$ terms in equation (\ref{hd-040}) as follows 
\begin{eqnarray}
\label{hdd-002}
\bF_3(|z|^4|v|^3): &=& (4 c_{i \bar m}e_{j\bar k,l} + 2 c_{i\bar m}\bar e_{j\bar k, l} +2c_{j\bar k,l} e_{i\bar m} + 2c_{j\bar k, l}\bar e_{i\bar m})v^i \bar v^k v^j
\nonumber\\
&+& (2 c_{q\bar k}e_{j\bar q,l} + c_{q\bar k,l}e_{j\bar q} + c_{j\bar q,l} \bar e_{q\bar k} + c_{j\bar q}\bar e_{q\bar k,l}) v^j\bar v^k v^m
\nonumber\\
&+& F_{3,vol},
\end{eqnarray}
where 
$$F_{3, vol}: = (4c_{k\bar k}e_{j\bar q,l} + 2c_{k\bar k}\bar e_{j\bar k,l} + 2c_{j\bar q, l}e_{k\bar k} + 2 c_{j\bar q,l}\bar e_{k\bar k} ) v^j \bar v^q v^m.  $$
And we have 
\begin{eqnarray}
\label{hdd-003}
\bF_2: &=& ( 2 c_{k\bar k,l} e_{j\bar m} + 2 c_{k\bar k,l} \bar e_{j\bar m} + 4c_{j\bar m} e_{k\bar k, l} + 2 c_{j\bar m}\bar e_{k\bar k, l}    )v^j
\nonumber\\
&+& ( 2 c_{k\bar m} e_{j\bar k, l} + c_{j\bar k,l} \bar e_{k\bar m} + c_{j\bar k}\bar e_{k\bar m, l} + c_{k\bar m, l}e_{j\bar k})v^j
\nonumber\\
&+& F_{2,vol} -2 v^m \d_l E_{vol},
\end{eqnarray}
where 
$$F_{2,vol}: = (2c_{j\bar m,l} e_{k\bar k} + 4 c_{k\bar k}e_{j\bar m,l} + 2c_{j\bar m,l}\bar e_{k\bar k} + 2 c_{k\bar k}\bar e_{j\bar m,l})v^j $$

Next we are going to apply a slight different trick with the case only involving lower order terms. 
It is necessary to compute the derivatives of the following 
\begin{eqnarray}
\label{hdd-004}
&&(2c_{j\bar q, l} -2\bar e_{j\bar q,l} -4 e_{j\bar q, l}) \frac{\d^2}{\d v^q \d \bar v^m} (\chi' v^j)
\nonumber\\
&=& r^{-4}\chi'''(L_3 + \bE_3 + \bH_3)  + r^{-2}\chi'' (L_2 +\bE_2 + \bH_2),
\end{eqnarray}
where $L_3, L_2$ are the same lower order terms appearing in equation (\ref{hd-043}),
and higher order terms are 
\begin{eqnarray}
\label{hdd-005}
\bH_3: &=& (2c_{j\bar q, l}\bar e_{q\bar k} + 2c_{j\bar q,l}e_{q\bar k} + 2c_{q\bar k} \bar e_{j\bar q,l} + 4 c_{q\bar k}e_{j\bar q,l}) v^j\bar v^k v^m
\nonumber\\
&+& (2 c_{j\bar q,l} \bar e_{i\bar m} + 2 c_{j\bar q,l}e_{i\bar m} + 2c_{i\bar m}\bar e_{j\bar q,l} +4 c_{i\bar m}e_{j\bar q,l}) v^i\bar v^q v^j,
\end{eqnarray} 
and 
\begin{eqnarray}
\label{hdd-006}
\bH_2: &=& (2c_{j\bar q, l}\bar e_{q\bar m} + 2c_{j\bar q,l}e_{q\bar m} + 2c_{q\bar m}\bar e_{j\bar q,l} +4 c_{q\bar m}e_{j\bar q,l}) v^j
\nonumber\\
&+& ( 2c_{k\bar k, l}\bar e_{j\bar m} + 2c_{k\bar k,l}e_{j\bar m} + 2c_{j\bar m}\bar e_{k\bar k, l} + 4 c_{j\bar m}e_{k\bar k, l})v^j
\end{eqnarray}
Subtract equation (\ref{hdd-004}) from equation (\ref{hd-040}), and then we have  
the following higher order terms left in the bracket of $r^{-4}\chi'''$
\begin{eqnarray}
\label{hdd-007}
\bF'_3: &=& F_{j\bar k, l} v^j \bar v^k v^m 
\nonumber\\
&=& v^j \bar v^k v^m \Big( -c_{j\bar q,l}\bar e_{q\bar k} -2 c_{j\bar q, l}e_{q\bar k} -2 c_{q\bar k}e_{j\bar q,l}
\nonumber\\
&+& c_{q\bar k,l}e_{j\bar q}+ c_{j\bar q}\bar e_{q\bar k,l} -2 c_{q\bar k}\bar e_{j\bar q,l} \Big) + F_{3,vol},
\end{eqnarray}
and in the bracket of $r^{-2}\chi''$
\begin{eqnarray}
\label{hdd-008}
\bF'_2 :&=&  v^j \Big( -2 c_{k\bar m}e_{j\bar k, l} - c_{j\bar q,l}\bar e_{q\bar m} + c_{j\bar k}\bar e_{k\bar m,l}
\nonumber\\
&+&c_{k\bar m,l}e_{j\bar k} -2 c_{j\bar q,l }e_{q\bar m} -2 c_{q\bar m}\bar e_{j\bar q,l} \Big) + F_{2,vol} - 2v^m \d_l E_{vol}.
\end{eqnarray}
Then a straightforward calculation shows the follows 
\begin{eqnarray}
\label{hdd-009}
 && (r^{-4}\chi''' ) \bF'_3 + ( r^{-2} \chi'' ) \bF'_2
\nonumber\\
& = & \frac{\d^2}{\d v^ q \d \bar v^m} \Big\{ v^j F_{j\bar q, l} \chi' \Big\}
\end{eqnarray}

Therefore, we have another equation to replace equation (\ref{hd-047})
\begin{eqnarray}
\label{hdd-010}
\nabla'_{\bar m}\nabla_l \{\chi' d\lambda\}&=& \frac{\d^2}{ \d v^q \d \bar v^m} \Big\{ ( 2 c_{j\bar q, l}  -2\bar e_{j\bar q,l} -4 e_{j\bar q, l} )v^j \chi' \Big\} d\lambda(u)
\nonumber\\
&-&   \frac{\d^2}{ \d v^q \d \bar v^m} \Big \{(2B_{j\bar ql, \bar p r\bar s} + \tilde A_{j\bar ql, \bar pr\bar s}) \bar z^p z^r \bar z^s v^j \chi' \Big\} d\lambda(u)
\nonumber\\
&+&  \frac{\d^2}{\d v^q \d\bar v^m}  \Big\{  v^j F_{j\bar q,l } \chi' \Big\} d\lambda (u)
\nonumber\\
&-& r^2 \frac{\d^2}{\d v^k \d\bar v^j} \Big\{  (c_{j\bar kl\bar m } - 2 e_{j\bar k l\bar mr}z^r - 2 e_{j\bar k\bar ml\bar s} \bar z^s) \chi  \Big\} d\lambda(u)
\nonumber\\  
&+& \cR'(z,v) d\lambda(u),
\nonumber\\
\end{eqnarray}
where the remaining error is
$$ \cR'(z,v): = \frac{\chi'''}{r^4} O(|z|^5|v|^3, |z|^2|v|^4)  + \frac{\chi''}{r^2} O(|z|^5|v|, |z|^2|v|^2) + \chi' O(|z|^2). $$

\subsection{Integration by parts}
Integration by parts gives us the following complex Hessian of $\tilde\phi$ at a point $z\in U''$ with convolution radius $r$. 
Recall that $u = v +z$, and we have 
\begin{eqnarray}
\label{ip-001}
\frac{\d^2 \tilde\phi}{\d z^l \d \bar z^m} (z_0,r)  s^l \bar s^m &=& \frac{1}{r^{2n}} \int \chi' \frac{\d^2\phi}{\d u^l \d \bar u^m } s^l\bar s^m d\lambda(u)
\nonumber\\
&+& \frac{1}{r^{2n}}\Re\int \chi' \frac{\d^2\phi}{\d u^l \d \bar u^m}   \Big\{ P_{j\bar l q} v^j \Big\} s^q\bar s^m d\lambda(u)
\nonumber\\
&-& \frac{1}{r^{2n-2}} \int \chi \frac{\d^2\phi}{\d u^k \d\bar u^j}  (c_{j\bar kl\bar m } -\cE)  
s^l\bar s^m d\lambda(u)
\nonumber\\
& + &\frac{1}{r^{2n}}\int \phi \cR'(z,v) d\lambda(u),
\nonumber\\
\end{eqnarray}
where 
\[
\cE: =  2 e_{j\bar k l\bar mr}z^r + 2 e_{j\bar k\bar ml\bar s} \bar z^s;
\]
and 
\[
P_{j\bar l q}: =  (2 c_{j\bar l, q}  -2\bar e_{j\bar l,q} -4 e_{j\bar l, q}) 
- (2B_{j\bar lq, \bar p r\bar s} + \tilde A_{j\bar lq, \bar pr\bar s}) \bar z^p z^r \bar z^s 
+   F_{j\bar l, q}.
\]
Here the tensors $F_{j\bar l, q}$ are defined in equation (\ref{hdd-007}), and they are all of order $|z|^4$. 
Moreover, these tensors also satisfy the following relation, i.e. 
$\ol{P_{j\bar kl }} = P_{\bar j k \bar l}$,
since all tensors involved like $c_{j\bar l r\bar s}$ and $e_{j\bar k l\bar p r}$ are bi-Hermitian.

Now if we choose the size of $r$ to be $\delta^3$, i.e. $r=O(\delta^3)$,
then it is easy to see $ \cR'(z,v) = O(|z|^2)$.
Therefore, the last error term in equation (\ref{ip-001}) is controlled by 
\begin{equation}
\label{ip-002}
|z|^2 \fint (-\phi) \leq  -C|z|^2\log \delta,
\end{equation}
for some uniform constant $C>0$. Obviously, this error converges to zero when $\delta$ does. 
Moreover, the integral of the error term is always bounded by $|z|^2\log\delta^{-1}$.

In fact, we have a rather simple formula at the center of the ball $U_{\a}$ for each $\a$, 
and then equation (\ref{ip-001}) becomes as follows 
\begin{eqnarray}
\label{ip-003}
\frac{\d^2 \tilde\phi}{\d z^l \d \bar z^m} (0,r)  s^l \bar s^m &=& \frac{1}{r^{2n}} \int \chi' \frac{\d^2\phi}{\d u^l \d \bar u^m } s^l\bar s^m d\lambda(u)
\nonumber\\
& -&\frac{1}{r^{2n-2}} \int \chi \frac{\d^2\phi}{\d u^k \d\bar u^j} c_{j\bar kl\bar m} s^l\bar s^m d\lambda(u).
\end{eqnarray}
The first term on the RHS of equation (\ref{ip-003}) is coming from the local convolution on Euclidean ball.
The second term is essentially induced from the curvature of the metric on the manifold,
and it is controlled by the Lelong number of $\varphi$ at this point. 

We can take a closer look at this term controlled by Lelong numbers. 
First we introduce the following lower and upper bound of the curvature tensor in the coordinate ball $U_{\a}$
$$ m_{\a}: = \inf_{|\z| = |\xi| =1} c_{j\bar kl\bar m} \z^j\bar \z^k \xi^l \bar\xi^m,  $$ 
and also  
$$ M_{\a}: = \sup_{|\z|= |\xi|=1} c_{j\bar kl\bar m} \z^j\bar \z^k \xi^l\bar\xi^m. $$
Put $m^-_{\a}: = \max \{ 0, - m_{\a} \}$ and $M^+_{\a}: = \max \{0, M_{\a} \}$, 
and the Lelong number $\nu(\phi)$ of the plurisubharmonic function $\phi$ at point $z_0$ is given by the following limit 
$$\nu(\phi)(z_0): = \lim_{r\rightarrow 0} \frac{1}{r^{2n-2}}\int_{B_r(z_0)} \Delta \phi(u)d\lambda(u), $$
and it is independent of the chosen coordinate. 
Put  $u = v+z_0$, and then we have 
\begin{lemma}
\label{ip-lem-001}
Let $\lambda_r, \Lambda_r$ be the following numbers at the point $z_0$
\begin{equation}
\label{ip-004}
\lambda_{r,\a}: = - \frac{(1+\delta)m_{\a}}{r^{2n-2}}\int \Delta \phi(u) \chi\left( \frac{|v|^2}{r^2} \right) d\lambda(u) + C \delta^2;
\end{equation}
\begin{equation}
\label{ip-005}
\Lambda_{r,\a}:  = - \frac{(1+\delta)M_{\a}}{r^{2n-2}}\int\Delta \phi(u) \chi\left( \frac{|v|^2}{r^2} \right) d\lambda(u) - C \delta^2,
\end{equation}
for uniform constant $C>0$. 
Then we have for all $z_0\in U''_{\a}$ and any $r$ small enough 
\begin{eqnarray}
\label{ip-006}
\Lambda_{r,\a} |s|^2 \leq \frac{1}{r^{2n-2}} \int \frac{\d^2\phi}{\d u^k \d\bar u^j} (c_{j\bar k l\bar m} - \cE) s^l\bar s^m \chi\left( \frac{|v|^2}{r^2}\right) d\lambda(u)
\leq  \lambda_{r,\a} |s|^2.
\nonumber\\
\end{eqnarray}
Moreover, the positive number $\Lambda_{r,\a} (\lambda_{r,\a})$ is increasing(decreasing) 
to the value $M_{\a}\cdot \nu_{\varphi}(z_0) (m_{\a} \cdot \nu_{\varphi}(z_0))$ as $r$ converging to zero.
\end{lemma}
\begin{proof}
First we can assume $\phi$ is plurisubharmonic locally by adding an error term like $C|z|^2$.
Then notice that the following matrix is Hermitian
$$C_{j\bar k}: = c_{j\bar k l\bar m}s^l \bar s^m,$$
and the largest(lowest) eigenvalue of $(C_{j\bar k})$ is bounded by $M_{\a}|s|^2 (m_{\a}|s|^2)$ by definition.
Suppose $F^{\bar q p}$ is another positive Hermitian matrix, and then we have 
\begin{equation}
\label{ip-007}
m_{\a} |s|^2\cdot tr(F) \leq    F^{\bar kj} C_{j\bar k} \leq M_{\a} |s|^2 \cdot tr(F),
\end{equation}
by a well known lemma in linear algebra. 
Then the rest part of the proof is following from Demailly's work \cite{Dem83} by putting
$$ F^{\bar k j} = \frac{\d^2\phi}{ \d\bar u^j \d u^k}. $$ 
\end{proof}

Now observe that all higher order terms convoluting with the Hessian of $\phi$ in equation $(\ref{ip-001})$ are also uniformly controlled by small errors.
For instance, if we pick up $r=O(\delta^3)$, then all $P_{j\bar l q\bar m}v^j$ terms are in the order of $O(\delta^4)$.
Therefore, we can assume our potentials function $\phi_j$ is a local $psh$ function by adding an error term like $(1+ O(\delta^4))|z|^2$.
Moreover, we can switch equation (\ref{ip-001}) to the following form 
\begin{eqnarray}
\label{ip-008}
\frac{\d^2 \tilde\phi}{\d z^l \d \bar z^m} (z_0,r)  s^l \bar s^m &=& \frac{1}{r^{2n}} \int \chi' \frac{\d^2\phi}{\d u^l \d \bar u^m }S^l\bar S^m d\lambda(u) - \cH
\nonumber\\
&-& \frac{1}{r^{2n-2}} \int \chi \frac{\d^2\phi}{\d u^k \d\bar u^j}  (c_{j\bar kl\bar m } -\cE)  
s^l\bar s^m d\lambda(u)
\nonumber\\
& + &\frac{1}{r^{2n}}\int \phi \cR'(z,v) d\lambda(u),
\nonumber\\
\end{eqnarray}
where 
$$\cH:=   \frac{1}{r^{2n-2}} \int \chi' \frac{\d^2\phi}{\d u^l \d \bar u^m } \Big\{ P_{j\bar l q} P_{\bar k m \bar p} (r^{-1}v^j) (r^{-1}\bar v^k)   \Big\}s^q\bar s^p d\lambda(u).$$
However, as we already observed, all the tensors like 
$$\cH_{q\bar l m\bar p}: = P_{j\bar l q} P_{\bar k m \bar p} (r^{-1}v^j) (r^{-1}\bar v^k) = O(|z|^2) $$ 
satisfy the bi-Hermitian relation,
and the a similar argument as in the proof of Lemma (\ref{ip-lem-001}) 
implies that $\cH$ is indeed controlled by the Lelong number of $\phi$ with a multiple of order $O(|z|^2)$! 

In conclusion, we proved that $\tilde\phi_j$ is a local quasi-$psh$ function on each $U''_j$ as 
\begin{equation}
\label{ip-009}
(1+ O(\delta^4)) \omega + i\ddbar \tilde\phi_j \geq   -( 1+ O(\delta^2) ) \lambda_{r,j}\omega. 
\end{equation}

Therefore, combing with the twisted term $h(\delta,z)|z|^2$ introduced in last section, 
we see that our glueing target $\Phi_j$ is also a local quasi-$psh$ function as 
\begin{equation}
\label{ip-010}
(1+ \lambda_{r,j} + O(\delta^2\log\delta^{-1}))\omega + i\ddbar \Phi_j \geq 0.
\end{equation}

Up to this stage, we proved the statement (1) in Theorem (\ref{thm-002}).



\section{Preserve smoothness}
Instead of using maximum operator, we are going to invoke the regularized maximum operator for glueing purpose.
However, this causes further regularity issues since the second derivative of a (regularized) maximum operator 
is blowing up in certain direction.

\subsection{Regularized maximum operator}
The regularized maximum operator $\cM^{(p)}_{\tau}: \bR^p \rightarrow \bR$ defined for any small $\tau>0$
is a smooth and convex function, which is also non-decreasing for each variable \cite{abook}. 
It is indeed a smoothing of the maximal function on $\bR^p$. 

In fact, on an ordered set, we can repeat applying $\cM^{(2)}_{\tau}$ for $p-1$ times to recover $\cM_{\tau}^{(p)}$,
i.e. we can define 
$$ \cM^{(p)}_{\tau}(x_1, x_2, \cdots, x_p): = \cM^{(2)}_{\tau} \Big( \cdots \cM_{\tau}^{(2)}\{ x_3, \cM^{(2)}_{\tau}(x_1, x_2) \} \Big). $$
Therefore, we will assume $p=2$ in the rest part of the paper. 
Then on the intersection of $U''_j$ and $U''_k$, we have to investigate the complex Hessian of $\cM_{\tau}(\Phi_{j}, \Phi_k)$. 

Put $2 y_1: = (x_1 + x_2) $ and $2 y_2: =(x_1 -x_2)$, and then we have the following result (Lemma (5.1), \cite{Li15}) 
\begin{eqnarray}
\label{rmo-001}
&&\d_{y_1}\cM_{\tau} = 1; \ \ \  | \d_{y_2} \cM_{\tau}| \leq 1;
\nonumber\\
&& \d^2_{y_1}\cM_{\tau} = 0; \ \ \  \d_{y_1}\d_{y_2}\cM_{\tau} = 0.
\end{eqnarray}
Moreover, the second derivative of $\cM_{\tau}$ in pure $y_2$ direction can be estimated as 
$$ \d^2_{y_2} \cM_{\tau} = O(\tau^{-1}). $$
Then we can calculate the complex Hessian on $w$-coordinate as follows 
\begin{eqnarray}
\label{rmo-002}
&&\frac{\d^2 }{\d w^{\a}\d \bar w^{\b}} \cM_{\tau} (y_1(\Phi_j, \Phi_k), y_2(\Phi_j, \Phi_k))
\nonumber\\
&=& \frac{\d^2 (\Phi_j + \Phi_k) }{2 \d w^{\a} \d\bar w^{\b}}   + \d_{y_{2}}\cM_{\tau}\cdot \frac{\d^2 (\Phi_j - \Phi_k)}{2 \d w^{\a} \d\bar w^{\b}}
\nonumber\\
&+& \d^2_{y_2} \cM_{\tau} \frac{\d (\Phi_j - \Phi_k) }{\d w^{\a}} \cdot \frac{\d (\Phi_j - \Phi_k)}{\d \bar w^{\b}}. 
\end{eqnarray}
First notice that the last term on RHS of equation (\ref{rmo-002}) forms a positive Hermitian matrix at a given point. 
Therefore, the lower bound of the complex Hessian will not be affected when one passes from local to global. 
Combing with equation (\ref{ip-010}), we proved the statement (2) of Theorem (\ref{thm-002}).

However, we have to take care of this term while considering the upper bound of the Hessian. 
Notice that we have 
$\d (\Phi_j - \Phi_k) = \d (\tilde\phi_j - \tilde\phi_k) + O(\delta^3\log\delta^{-1})$,
and then we only need to compare the derivatives of $\tilde\phi_j$ and $\tilde\phi_k$ at a fixed point $p\in U''_j \cap U''_k$.
This is essentially the same estimate as we did in proving equation (\ref{hd-0011}),
and we claim that the following estimate holds
\begin{equation}
\label{rmo-0002}
|\d (\tilde\phi_j -\tilde\phi_k )| \leq C \delta\log \delta^{-1},
\end{equation}
for some uniform constant $C$.

The proof of this claim will be postponed to next section, 
and then the last term on RHS of equation (\ref{rmo-002}) can be controlled by choosing $\tau = ( -\log \delta )^{-1} $ as 
\begin{equation}
\label{rmo-0003}
\tau^{-1} \frac{\d (\Phi_j - \Phi_k) }{\d w^{\a}} \cdot \frac{\d (\Phi_j - \Phi_k)}{\d \bar w^{\b}} \leq C\delta^2(-\log\delta)^3,
\end{equation}
for some uniform constant $C>0$. 

Therefore, combing equation (\ref{rmo-001}), (\ref{rmo-002}) and (\ref{rmo-0003}), we proved our statement $(3)$ in Theorem (\ref{thm-002}).

\subsection{Estimates on derivatives}
We will prove the claim, equation (\ref{rmo-0002}), in this section. 
First there always exists a sequence of local smooth (quasi-)$psh$ functions $\phi_{\ep}$ to approximate $\phi$ around the point $p$.
For instance, in a small geodesic ball centered at $p$, take the convolution $\phi * \rho_{\ep}$ of $\phi$ w.r.t. a mollifier $\rho_{\ep}$ supported on this ball. 
Then $\phi_{\ep}$ is again a (quasi-)$psh$ function, decreasing to $\phi$ pointwise, and converges to $\phi$ in $L^p$ for any $p>0$ on this ball. 

Put $\phi_{j,\ep}: = \phi_{\ep}\circ \tau_j^{-1}$ and $\phi_{k,\ep}: = \phi_{k,\ep}\circ \tau_k^{-1}$,
and then it is easy to see that we have $\phi_{k,\ep} = \phi_{j,\ep}\circ \tau$ again. 
As before, we assume $z_0 = \tau (w_0)$ corresponds to the point $p$ in the intersection.
Now we can also define their convolutions as 
$$ \tilde\phi_{j,\ep}(z,r): = r^{-2n}\int_{\z\in\bC^n} \phi_{j,\ep}(z+ A^{-1}(z) \cdot \z) \chi'(r^{-2n}|\z|^2) d\lambda(\z),$$
and $\tilde\phi_{k,\ep}$ is defined in a similar way. 
Next we claim that the derivative $\d \tilde\phi_{j,\ep}(z_0,r)$ converges to $\d\tilde\phi_{j}(z_0,r)$ 
while $\ep\rightarrow 0$, for any fixed point $z_0$ and radius $r$.  
The reason is as follows:
first we can put $u: = z+ v$ and $v: = A^{-1}(z)\z$ as before,
and then view $\z(z,v)$ as a function of $z$ and $v$ variables.  
Thus the derivative of $\tilde\phi_{j,\ep}$ at the point $z_0$ can be computed as  
\begin{equation}
\label{rmo-003}
\d_{z} \tilde\phi_{j,\ep} = \frac{1}{r^{2n}}\int \phi_{j,\ep} (z_0 + v) \d_{z} G_r(z_0,v) d\lambda(v),
\end{equation}
where 
$$G_r(z,v): =  \chi'(r^{-2} |\z(z,v)|^2)\cdot \frac{d\lambda(\z)}{d \lambda(v)}$$
is a smooth cut-off function supported in a small ball around $(z_0,0)\in \bC^{n}\times \bC^n$. 
Then our claim is clear from the definition of the convolution. 

Therefore, in order to control $\d (\tilde\phi_j - \tilde\phi_k)$, it is enough to have a uniform estimate 
on $\d (\tilde\phi_{j,\ep} - \tilde\phi_{k,\ep})$. 
However, notice that we need to take derivatives of $\tilde\phi_{j,\ep}$ on $w$-coordinate
\begin{equation}
\label{rmo-005}
\frac{\d \tilde\phi_{j,\ep} }{\d w^{\a}} (w_0,r) = \Big(\d\tau(w_0) \Big)^{\mu}_{\a} \left\{ \frac{1}{r^{2n}}\int \frac{\d\phi_{j,\ep}}{\d z^{\mu}} (z_0,\z) 
\chi' \left(\frac{|\z|^2}{r^2}\right)d\lambda(\z) \right\},
\end{equation}
where 
\begin{equation}
\label{rmo-006}
\frac{\d \phi_{j,\ep}}{\d z^{\mu}}(z,\z) = \d_{\mu}\phi_{j,\ep}(u) (1 + O(\delta)\z).
\end{equation}
Moreover, we can rewrite $\tilde\phi_{k,\ep}$ as before, and its derivatives at $w_0$ is 
\begin{eqnarray}
\label{rmo-007}
\frac{ \d \tilde\phi_{k,\ep}}{\d w^{\a}} (w_0,r) &=& \frac{\d}{\d w^{\a}} \left\{  \frac{1}{r^{2n}}\int \phi_{j,\ep}\circ \tau (w + B^{-1}(w)\cdot \z) \chi' d\lambda (\z) \right\} 
\nonumber\\
&=&  \frac{1}{r^{2n}} \int \d_{\mu}\phi_{j,\ep} (\tau (u'))  \Big( \d\tau(u') \Big)^{\mu}_{\a} (1 + O(\delta)) \chi' d\lambda,
\end{eqnarray}
where $u' = w_0 + B^{-1}(w_0)\cdot\z $.
However, thanks to Lemma (\ref{hd-lem-005}), we have the following estimate for these holomorphic functions 
\begin{eqnarray}
\label{rmo-008}
( \d\tau )^{\mu}_{\a} (w+ \eta) - (\d\tau)^{\mu}_{\a} (w) &=& (\d^2\tau)^{\mu}_{\a\gamma} \eta^{\gamma} + O(|\eta|^2)
\nonumber\\
&=& O(\delta |\eta|, |\eta|^2). 
\end{eqnarray}
Moreover, on compact K\"ahler manifold, there exists a uniform constant $C$, such that we have
\begin{equation}
\label{rmo-009}
\fint_{B(r)} |\nabla\phi | \leq C r^{-1}\log r^{-1},
\end{equation}
in any small ball $B(r)$ around each point $p\in X$. 
Therefore, with an error term uniformly controlled by $O(\delta\log\delta^{-1})$, 
it is enough to estimate the difference between 
\begin{equation}
\label{rmo-010}
\int \d_{\mu}\phi_{j,\ep}(z_0 +r A^{-1}(z_0)\z ) \chi' d\lambda 
\end{equation}
and
\begin{equation}
\label{rmo-0010}
 \int \d_{\mu}\phi_{j,\ep}\circ\tau (w_0 + rB^{-1}(w_0)\z) \chi' d\lambda.
\end{equation}

However, this is exactly what we did in proving equation (\ref{hd-0011}).
Then a similar argument by replacing $\phi_j$ by $\d\phi_{j,\ep}$ gives the following estimate  
\begin{equation}
\label{rmo-011}
| ( \ref{rmo-010}) -(\ref{rmo-0010}) | \leq C' (\delta + r) \log r^{-1}.
\end{equation}
Here the constant $C'>0$ only depends on the geometry of $X$ and $\varphi$. 
Therefore, the difference is uniformly (not depending on $\ep$) controlled as 
\begin{equation}
|\d ( \tilde\phi_{j,\ep} -\tilde\phi_{k,\ep} ) | \leq C'' \delta\log\delta^{-1},
\end{equation}
and then our claim follows.

\section{Obstruction}
One may expect that the trick of doing integration by parts always works for different higher order terms in equation (\ref{hd-040}).
Unfortunately, this is not true even in one dimension! 
In fact, the obstruction of switching to integration by parts comes from those ``good terms'' like all $\bD_i$ terms in the equation. 
These terms are ``goood'' in the sense 
that their averages on the ball $B(r)$ are uniformly controlled by $\delta^2$ while $r\rightarrow 0$ and $z_0$ fixed. 
We will show this phenomenon in one dimension below. 

Suppose $X$ is a Rieman surface, then we can truncate the metric $\omega_g$ on a small normal coordinate ball around a point $p\in X$ as
\begin{equation}
\label{rp-002}
g(z) = 1 - c|z|^2 + O(|z|^3),
\end{equation}
where the tensor $c$ corresponds to the curvature of the metric $\omega_g$ at the origin. 
Put 
$$a(z): = 1 - \frac{1}{2}c|z|^2, $$
and we are going calculate the complex Hessian of the following convolution
$$\tilde\phi(z,r) = \frac{1}{r^2}\int_{\eta\in\bC} \phi(z+ a^{-1}(z)\cdot \eta)\chi'(r^{-2}|\eta|^2)d\lambda(\eta), $$
on a local coordinate ball $U$. Put $u =z + a^{-1}\eta$ and $v = u-z$, and then we have 
\begin{equation}
\label{ob-002}
\eta = v - \frac{1}{2} c|z|^2v.
\end{equation}

Denote the following two differential operators for the commutator as before: 
$$\nabla = \frac{\d}{\d z} + \frac{\d}{\d u};\ \ \ \ \ol\nabla' = \frac{\d}{\d\bar z} - \frac{\d}{\d\bar u}, $$
and then we have $\nabla v =0$, $\nabla \bar v =0$ and $\ol\nabla' v =0$, $\ol\nabla' \bar v = -2$.
Now a similar computation shows that we have the following equation 
\begin{eqnarray}
\label{ob-003}
\ol\nabla'\nabla(\chi' d\lambda) &=& d\lambda(u) \Big\{ r^{-4}\chi''' |v|^2( 2c\bar zv - 5 c^2 |z|^2 \bar z v 
\nonumber\\
&+&  c^2|z|^2|v|^2 +O(|z|^5 |v|, |z|^4|v|^2) )
\nonumber\\
&+& r^{-2} \chi'' ( - c|v|^2 + 4 c\bar zv   -   6 c^2 |z|^2 \bar z v 
\nonumber\\
&+& 4 c^2|z|^2|v|^2 + O(|z|^4|v|^2, |z|^5|v|))
\nonumber\\
&+& \chi' (-c  + c^2 |z|^2) \Big\}.
\end{eqnarray}
Moreover, notice that the following modification has no effect on those $\bD_i, (i=1,2,3)$ in equation (\ref{ob-003})
\begin{eqnarray}
\label{ob-004}
\frac{\d^2}{\d v\d\bar v} \{2c\bar z v\chi' (r^{-2}|\eta|^2) \} &=& r^{-4}\chi''' |v|^2 (2c\bar zv - 4c^2|z|^2\bar z + O(|z|^5 |v|))
\nonumber\\
&+& r^{-2}\chi'' (4c \bar zv - 4 c^2|z|^2\bar zv  + O(|z|^5|v|)). 
\nonumber\\
\end{eqnarray}

Therefore, $\bD_i$ terms are left in equation (\ref{ob-003}) as 
\begin{eqnarray}
\label{ob-005}
&& r^{-4}\chi'' ( c^2 |z|^2 |v|^4) + 4 r^{-2}\chi'' ( c^2 |z|^2 |v|^2) + \chi' (c^2|z|^2)
\nonumber\\
&=& \frac{\d^2}{\d v\d\bar v} \{c^2|z|^2 \chi' |v|^2 \} + r^{2} \frac{\d^2}{\d v\d\bar v} \{ c^2|z|^2 \chi  \} - \chi' (c^2|z|^2).
\end{eqnarray}
Unfortunately, it seems there is no way to deal with the last term on the RHS of equation (\ref{ob-005}),
and we always left some terms with error like 
$$|z|^2 \fint_{B(r)}  (-\phi),$$ 
in the integral. 

In fact, there is a way to cancel these terms. 
Remember we only did a linear change of the variables as $z\rightarrow z+ a^{-1}\eta $,
and then the point is that we can add some quadratic terms of $\eta$ in the twisting as 
$$\eta = v - \frac{1}{2} \bar z (z+v) v,$$
and then we can hope to play the same trick of integration by parts to these $\bD_i$ terms. 
This will enable us to improve the error term in the statement (1) of Theorem (\ref{thm-002}) to $O(|z|^4\log\delta^{-1})$.

\section{Appendix }
\label{app-B}
Let $A$ be a hermitian matrix, and $B_A(z,r)$ denote a twisted Euclidean ball in $\bC^n$ by
$$B_A(z,r): = \{ z+rA\cdot \z;\ \ \ |\z|<1 \}. $$
If the center is the origin, then we simply use $B_A(r)$ instead of $B_A(0,r)$. 

Suppose $\varphi$ is a bounded plurisubharmonic function in a domain $\Omega \subset \bC^n$,
and $\rho(z) = \rho (|z|)$ is a standard mollifier supported in the Euclidean unit ball $B$. 
We can further define the following three functions in a slightly smaller domain $\Omega'$ 
by writing 
\begin{equation}
\label{B-001}
\phi_A(z,t):= \fint_{\d B} \varphi (z+ e^t A\cdot \z) d\sigma(\z),
\end{equation}

\begin{equation}
\label{B-002}
\Phi_A(z,t) : = \sup_{\z\in B}\varphi (z+ e^t A\cdot \z),
\end{equation}

\begin{equation}
\label{B-003}
\tilde{\phi}_A (z,t): = \int \varphi(z + e^t A\cdot \z) \rho(|\z|) d\lambda(\z),
\end{equation}
Observe that we have 
\begin{equation}
\label{B-0-003}
\tilde{\phi}_A(z,t) = \int_{-\infty}^{0} \phi_A(z,t + s) \hat{\rho}(s) ds,
\end{equation}
where $$ \hat{\rho}(s): = \sigma(\d B) e^{2ns}\rho(e^s).$$

\begin{lemma}
\label{B-lem-001}
For $t$ small enough, we can estimate the difference between $ \tilde\phi_A$ and $\Phi_A$ as 
\begin{equation}
\label{B-0-004}
0 \leq \Phi_A (z,t) - \tilde\phi_A (z,t) \leq C_A |t|^{-1}, 
\end{equation}
where $C_A$ is a uniform constant only depending on $osc(\varphi)$.  
\end{lemma}

For fixed matrix $A$, the translation $z \rightarrow z + A\cdot \z$ is linear.
Thanks to \cite{Hom},
the function $\varphi(z+ A\cdot\z)$ is plurisubharmonic both in $(z,\z)\in\Omega' \times \bC^n$. 
Therefore, if we define the following function by 
$$\varphi_A (\z) = \varphi_{A,z}(\z) = \varphi(z+A\cdot \z),$$
then $\varphi_A$ is plurisubharmonic in $\z\in \bC^n$. 
And we can re-write our three functions for fixed $z$ as 
\begin{equation}
\label{B-004}
\phi_A(t) = \fint_{\d B}  \varphi_{A} (e^t\z)d\sigma(\z),
\end{equation}

\begin{equation}
\label{B-005}
\Phi_A(t) = \sup_{\z\in B}\varphi_{A}(e^t\z),
\end{equation}

\begin{equation}
\label{B-006}
\tilde{\phi}_A (t): = \int \varphi_{A}(e^t\z)\rho(|\z|) d\lambda(\z). 
\end{equation}

This immediately implies that $\phi_{A}, \Phi_A, \tilde{\phi}_A$ are all non-decreasing and convex in $t$. 
Moreover, we can further estimate the convergence rate of these functions when $t$ goes to $-\infty$ as in \cite{BK}.
For given fixed $r\in (-\infty, 0)$, and $s\geq 0$, we have 
\begin{equation}
\label{B-007}
0 \leq \Phi_A (t+s) - \Phi(t) \leq \frac{s}{r -t} (\Phi_A(r) - \Phi_A(t)),
\end{equation}
for $t$ small enough. 
Notice that the RHS of equation (\ref{B-007}) converges to zero uniformly(Not depending on $z$!) by boundedness of $\varphi$
\begin{equation}
\label{B-008}
\frac{\Phi_A(r) - \Phi_A(t)}{r -t} \leq (r-t)^{-1} osc (\varphi) \leq O(|t|^{-1}).  
\end{equation}
Apple Harnack's inequality \cite{Kis}, we further see 
\begin{equation}
\label{B-009}
0\leq \Phi_A(t) - \phi_A(t)\leq \frac{3^{2n-1}}{2^{2n-2}} (\Phi_A (t) - \Phi_A(t- \log 2)),
\end{equation}
and 
\begin{eqnarray}
\label{B-010}
0 &\leq& \phi_A (t) - \tilde{\phi}_A(t) =\int_{-\infty}^{0} (\phi_A(t) - \phi_A(t+s)) \hat{\rho}(s) ds 
\nonumber\\
&\leq& \int_{-\infty}^{0} \frac{-s}{r - t-s } (\phi_A(r) - \phi_A(t+s))\hat\rho(s)ds
\nonumber\\
&\leq& O(|t|^{-1}),
\end{eqnarray}
by the same reason as in inequalities (\ref{B-007}) (\ref{B-008}). 
Finally, our result follows from combing equations (\ref{B-009}) and (\ref{B-010}).

\end{document}